\documentclass[12pt]{amsart}
\usepackage{verbatim}
\usepackage{amssymb}
\usepackage{amsrefs}

\usepackage[mathscr]{euscript}
\usepackage{enumerate}
\usepackage[active]{srcltx}
\numberwithin{equation}{section}

\usepackage{t1enc,xcolor}
\usepackage[utf8x]{inputenc}

\newtheorem{theorem}{Theorem}[section]
\newtheorem{lemma}[theorem]{Lemma}

\newtheorem{proposition}[theorem]{Proposition}

\newtheorem{claim}{Claim}[theorem]
\newtheorem{fact}{Fact}[theorem]

\theoremstyle{definition}
\newtheorem{definition}[theorem]{Definition}

\theoremstyle{remark}

\def\val{\mbox{val}}
\def\dom{\mathop{dom}}

\title{Automorphisms of $\mathcal P(\omega)/\mbox{fin}$ and large continuum}
\author{Alan Dow}
\address{Department of Mathematics,
University of North Carolina at Charlotte,
 Charlotte, NC 28223, USA}
\email{adow@uncc.edu}

\date{June 25, 2022}

\begin{document}

\begin{abstract} We prove that  it is consistent
with $\mathfrak c>\aleph_2$
that all automorphisms of $\mathcal P(\omega)/\mbox{fin}$
are trivial.      
\end{abstract}

\maketitle

\section{Introduction}
The study of automorphisms of $\mathcal P(\omega)/\mbox{fin}$
has, by now, an extensive and fascinating history.
  Naturally $\mathcal P(\omega)/\mbox{fin}$ is 
  the quotient Boolean algebra of the complete
  Boolean power set algebra $\mathcal P(\omega)$ by the
  ideal $\mbox{fin}$ of finite sets. 
   Every bijection between cofinite subsets
   of $\omega$ induces an automorphism of
   $\mathcal P(\omega) /\mbox{fin}$ and
   such automorphisms are said to be trivial.
 W. Rudin \cite{WRudin} established that the continuum hypothesis
 implied that there were non-trivial automorphisms.
 S. Shelah \cite{ShelahProper} established
 that it was consistent that all automorphisms were
  trivial and Shelah and Steprans \cite{ShSt88} 
  proved that this was a consequence of PFA. Our
  results follow the basic approach   of
  both \cites{ShelahProper,ShSt88}
 but also benefit from the considerable contributions
 in \cites{Vel92, VelOCA, Farah2000}  and others.
  Shelah and Steprans \cite{ShSt2001} have shown that it is consistent
with $\mathfrak c>\aleph_2$ there are non-trivial,
even nowhere trivial, automorphisms.  
   In this paper we establish that it is
   consistent with $\mathfrak c>\aleph_2$ 
   that all automorphisms are trivial.

 For a function $F:\mathcal P(\omega)\rightarrow   \mathcal P(\omega)$, 
  we say that $F$   induces an automorphism (of $\mathcal P(\omega)/\mbox{fin}$)
   if $F(x)=^*F(y)$ whenever $x=^* y$ 
   and the function sending the equivalence class
   of $x$ (mod finite) to that of $F(x)$ (mod finite)
   is indeed an automorphism of $\mathcal P(\omega)/\mbox{fin}$. 
   When $F$ does not induce a trivial automorphism,
   the family $\mbox{triv}(F)$ is the ideal
   of sets
   $a\subset\omega$ such that $F\restriction \mathcal P(a)$ is trivial
   in the usual sense, namely that there is a 1-to-1 
   function $h_a$ from  
   $a$ into $\omega$ 
   satisfying that $F(x) =^* h_a(x)$ for all $x\subset a$.
  We will, by default, let $h_a$ denote a 1-to-1
   function that induces $F$ on $\mathcal P(a)$ when $a$
   is an infinite element of $ \mbox{triv}(F)$. 
 It is easily seen that $\mbox{triv}(F)$ is an ideal
 on $\omega$ \cite{ShelahProper}. As usual,
  $\mbox{triv}(F)^+$ will denote (infinite)
  subsets of $\omega$ that are not elements
  of $\mbox{triv}(F)$.

    We adopt the $P$-name convention, for each poset $P$,
     that a $P$-name is a set of ordered pairs 
     where the first coordinate is a $P$-name
     and the second coordinate is an element of $P$. 
      We will also abuse notation in a standard way
      and treat ordinals and finite tuples
      of ordinals as $P$-names for themselves.
      Hence, for example, a subset of $\omega\times P$ will be 
      regarded as a $P$-name for a subset of $\omega$.
      We typically use the dot notation $\dot y$ to
      denote a $P$-name of a set. In the context
      of an argument with a generic filter in the
      discussion, removing the dot will denote
      the set that results from evaluating
      the name using that generic filter.

 \section{Tools for your forcing all automorphisms are trivial construction}

 \begin{proposition} If $F$ induces\label{badx}
 an automorphism and $a\in 
 \mbox{triv}(F)^+$, then for any 1-to-1 function 
 $h \in {}^a\omega$,
  there is an infinite $x\subset a$ such that
    $h(x)\cap F(x)$ is finite.
 \end{proposition}
 
 \begin{proof} Since $a\notin \mbox{triv}(F)$, the function
 $h$ does not induce $F\restriction\mathcal P(a)$. Choose
 any $x_1\subset a$ so that $h(x_1)\Delta F(x_1)$ is infinite.
 If $h(x_1)\setminus F(x_1)$ is infinite, then let
  $x\subset x_1$ be chosen so that $h(x)= h(x_1)\setminus F(x_1)$.
   Since $F(x)\subset^* F(x_1)$, we have $h(x)\cap F(x)$ is finite
   as required. In the other case when $F(x_1)\setminus h(x_1)$
   is infinite, choose $x\subset x_1$ so that 
    $F(x)=^* F(x_1) \setminus h(x_1)$. Here again we
    have that $h(x)\subset h(x_1)$ and $h(x_1)\cap F(x)=^*\emptyset$.
 \end{proof}

 \begin{definition} A family $\mathcal H$ of possibly partial functions
 from a countable set $D$  to $\omega$ is coherent if for $h_1,h_2\in \mathcal H$,
   the set $\{ n\in \dom(h_1)\cap \dom(h_2) : h_1(n)\neq h_2(n)\}$ is
   finite.\medskip
   
   Such a coherent family $\mathcal H$ is maximal
   if whenever $\mathcal H\cup \{\bar h\}$ is coherent,
    the domain of $\bar h$ is in the ideal generated by
     $[D]^{<\aleph_0}\cup \{ \dom(h) : h\in \mathcal H\}$.
     A coherent family will be called trivial if
      $D$ is in this ideal.
   
   \medskip

   For a function $f\in \omega^\omega$, let $f^\downarrow$ denote
   the set $\{ (n,m)\in \omega\times\omega : m < f(n)\}$. 
   A family $\mathcal H$ of functions is an $\omega^\omega$-family
   if $\mathcal H=\{ h_f : f\in \omega^\omega\}$ and, for
   each $f\in \omega^\omega$, $h_f$ is a function
   from $f^\downarrow $ to $\omega$. 
 \end{definition}
   
   We note that every subfamily of a coherent family is also
   coherent.   
   
    \begin{definition}
Say that the principle $\omega^\omega$-cohere  holds if
each $\omega^\omega$-family $\mathcal H$ that is coherent,
 there is a function $h:\omega\times\omega\rightarrow \omega$
 such that $\mathcal H\cup \{h\}$ is also coherent.
    \end{definition}
   
   The principle $\omega^\omega$-cohere (not so named)
   is a well-known consequence of OCA due
   to Todorcevic (see \cite{Farah2000}*{2.2.7}).
   This next result is similar to \cite{VelOCA}*{proof of Lemma 2.5}.
    
    \begin{lemma} Suppose that\label{notP}
    $F$ induces an automorphism on $\mathcal P(\omega)/\mbox{fin}$ and
    suppose that $\mbox{triv}(F)$ is proper  dense ideal.
    The principle $\omega^\omega$-cohere implies
    that if $\{a_n : n\in\omega\}$ is a 
    mod finite increasing sequence of
  subsets of $\omega$ 
     then either there is an $n$ so
     that $\omega\setminus a_n\in \mbox{triv}(F)$,
     or      
     there is an $b \in 
     \mbox{triv}(F)^+$ that is almost
     disjoint from each $a_n$.
    \end{lemma}
     
\begin{proof}
Let $\{a_n : n\in\omega\}\subset[\omega]^{\aleph_0}$ be
increasing. There is nothing to prove if the family
 $\{a_n : n\in\omega\}$ is eventually constant mod finite,
  so we may assume that it is strictly increasing
  mod finite.
As usual, let $\{a_n \}_n^\perp$ denote
the ideal of sets $b$ that are almost disjoint from
 each $a_n$. Assume that $b\in \mbox{triv}(F)$ for
 all $b\in \{a_n\}_n^\perp$, and as usual, let
  $h_b$ be the function on $b$ inducing $F\restriction \mathcal P(b)$.
 Since the family $\mathcal H=\{h_b : b \in \{a_n\}_n^\perp\}$
 is a coherent family, it is an obvious consequence
 of $\omega^\omega$-cohere that there is a function $h\in \omega^\omega$
 such that $\mathcal H\cup \{h\}$ is coherent. 
  
\medskip

  We show that there is an $n_0$ so that $h\restriction (\omega\setminus a_{n_0})$
 is 1-to-1. Otherwise we may choose a sequence
 of pairs $\{ (i_n, j_n) : n\in\omega\}$ so that
  $\{i_n,j_n\}\subset \omega\setminus a_n$ and
   $h(i_n)=h(j_n)$. By construction the set $b=\{i_n:n\in\omega\}\cup
   \{j_n:n\in\omega\}$ is an element of $\{a_n\}_n^\perp$. 
   Since $h_b$ is 1-to-1, this contradicts that $h_b\subset^* h$. 
   
\medskip    
   
    If there is an $n\geq n_0$ so that $h\restriction (\omega\setminus a_n)$
 induces $F$ on $\mathcal P(\omega\setminus a_n)$ then 
 the Lemma is proven. Otherwise, we may choose, using Lemma \ref{badx}, 
 for each $n_0\leq n\in\omega$ an infinite set $x_n \subset \omega\setminus
 a_n$ so that $h(x_n)\cap F(x_n)$ is finite. Since $x_n$
 cannot be an element of $\{a_n\}_n^\perp$, we may shrink $x_n$
 so that there is an $k_n\geq n$ such that $x_n\subset a
_{k_n+1}\setminus a_{k_n}$. Since $\mbox{triv}(F)$ is assumed
to be dense, we may also assume that $x_n\in \mbox{triv}(F)$.
Choose an infinite $J\subset\omega$ so that for $n<m\in J$,
 $k_{n+1}\leq m$. 
For each $n\in J$, let $h_n$ denote the function  $h_{x_n}$.
Note that $h(x_n) \cap h_n(x_n)=^*
h(x_n)\cap F(x_n)$ is finite. By a simple recursion we can 
further arrange that $h(i)\neq h_n(j)$  
for all $i\in \bigcup\{x_k : k\in J\}$
and $j\in x_n$. Let $x=\bigcup\{ x_n : n\in J\}$. 
For each $n\in J$, choose a finite $H_n\subset x_n$
so that $h_n(x_n\setminus H_n) \subset F(x)$. 
By recursion on $n\in J$,
 choose $i_n\in x_n\setminus H_n$ so that
  $ h_n(i_n) \notin \bigcup\{ F(a_k) : k
  \leq k_n\}$. 
Choose $y\subset \omega $ so that $F(y)=^* \{ h_n(i_n) : n\in J\}$
and note that $y\cap a_k$ is finite for all $k\in\omega$. 
Since $F(y)\subset F(x)$ we may assume that $y\subset x$.
But now, $F(y)=^* h(y)$ and yet $h(y)$ is disjoint
from $\{ h_n(i_n) : n\in J\}$. 
\end{proof}

     \begin{definition} Say that a family 
     $\{ h_\alpha :\alpha\in\omega_1\}$
     of partial functions from $\omega$ to $\omega$
     is 
     Luzin incoherent if for all
      $\alpha < \beta $,  $h_\alpha\cup h_\beta$ is
      not a function.
     \end{definition}
     
\begin{proposition} If $\mathcal H = \{h_\alpha : \alpha\in\omega_1\}$
is a Luzin incoherent family, then for all $h\in\omega^\omega$,
  $\mathcal H\cup \{h\}$ is not coherent.
\end{proposition}

\begin{proof} Let $h\in\omega^\omega$ 
and assume, for a contradiction, that $\mathcal H\cup \{h\}$
is coherent.
Choose $m\in\omega$
and function $t: m\rightarrow \omega$ so that,  there
is an uncountable $\Gamma\subset\omega_1$ so that,
for all $\alpha \in \Gamma$,
$h_\alpha\restriction(\dom(h_\alpha)\setminus m)\subset h$
and $h_\alpha\restriction m\subset t$. Of course we now
find that for all $\alpha < \beta$ both in $\Gamma$,
  $h_\alpha\cup  h_\beta(n)$ is a subfunction of 
   $t\cup h\restriction (\omega\setminus m)$. 
\end{proof}

 \begin{definition}
  An ideal $\mathcal I$ on a set $A\subset\omega$ is ccc over fin
  if  $\mathcal A\cap \mathcal I$ is not empty
  for 
  every uncountable almost disjoint family $\mathcal A$
  of
  subsets of $A$.
 \end{definition}
 
 The notion of an ideal being ccc over fin was
 introduced in \cite{Farah2000}.
  
     \begin{definition} Say that the principle 
     P-cohere holds
     providing there is no non-trivial
      coherent family
      of functions 
     whose domains form a ccc over fin $P$-ideal.
     \end{definition}
  
  \begin{lemma} Assume that the principles $\omega^\omega$-cohere
  and P-cohere hold. If $F$ induces\label{notcccoverfin} an automorphism
  on $\mathcal P(\omega)/\mbox{fin}$ then $A\in \mbox{triv}(F)$
  for any $A\subset\omega$ such that 
   $\mbox{triv}(F)\cap \mathcal P(A)$ is
    ccc over fin.
  \end{lemma}
  
  \begin{proof} Recall that we have assumed a fixed assignment
  of $h_a\in {}^a\omega$ for all $a\in \mbox{triv}(F)$ 
  such that $h_a(x)=^* F(x)$ for all $x\subset a$. 
  Since $F$ induces an automorphism, it follows that the
  family $\mathcal H = \{ h_a : a\in \mbox{triv}(F)\}$ is
  a coherent family. 
 Since a ccc over fin ideal is a dense ideal,
  we assume for the remainder of the proof that 
   $\mbox{triv}(F)$ is a proper dense ideal. We prove
   that it can not be ccc over fin.
  
  We first check that if
if $h\in {}^\omega\omega$, then $\mathcal H\cup \{h\}$ is
not coherent. Since we are assuming that
 $\omega\notin \mbox{triv}(F)$, we have, by Lemma \ref{badx},
 that there is an  infinite $x\subset\omega$ such 
 that $h(x)\cap F(x)$ is finite. We are also assuming
 that $\mbox{triv}(F)$ is a dense ideal, so, by possibly
 shrinking $x$, we can assume that $x\in\mbox{triv}(F)$.
 Then $h_x\in \mathcal H$ and evidently,
   $\{h_x,h\}$ is not coherent.
  
 It is immediate from the assumption
 that P-cohere holds that 
 $\mbox{triv}(F)$ is not a $P$-ideal.
 We complete the proof by considering the case where
  $\mbox{triv}(F)$ is not a P-ideal. Fix any
  increasing sequence $\{ a_n : n\in \omega\}\subset
   \mbox{triv}(F)$ with the property that
    $a\notin \mbox{triv}(F)$ for any $a\subset\omega$
    that mod finite contains each $a_n$. Since
    $\omega\notin\mbox{triv}(F)$, it follows
    that $\omega\setminus a_n\notin \mbox{triv}(F)$
    for all $n\in\omega$. 
    Recall that $\{a_n\}_n^\perp$ denotes the ideal
    of sets $b\subset \omega$ that are almost disjoint
    from $a_n$ for each $n$. For each $b\in \{a_n\}_n^\perp$,
     the ideal $\{b\cup a_n\}_n^\perp$ is a P-ideal.
     Using $\omega^\omega$-cohere
    and Lemma \ref{notP}, 
    for every $b_0\in\{a_n\}^\perp$, 
    there is a $b_1\in\mbox{triv}(F)^+\cap
    \{b_0\cup a_n\}_n^\perp    $. Therefore,
     we can by recursion, construct 
     a  mod finite increasing sequence
      $\{ b_\alpha : \alpha <\omega_1\}
      \subset \{a_n\}_n^\perp$ so that, for
      each $\alpha <\omega_1$,
        $b_{\alpha+1}\setminus b_\alpha
        \in 
        \mbox{triv}(F)^+\cap
    \{b_\alpha\cup a_n\}_n^\perp$. 
    This clearly shows that $\mbox{triv}(F)$
    is not then ccc over fin.  
   \end{proof}

\begin{definition} Let $\mathcal H = 
\{ h_\alpha :\alpha\in\omega_1\}$
 be a coherent family of partial functions on $\omega$.
 Define the poset $Q(\mathcal H)$ to be the
 set of pairs $q = ( \{ s^q_\rho : \rho\in 2^{\leq n_q}\}, \Gamma_q )$
  where
  \begin{enumerate}
  \item $n_q\in\omega$ and $\{ s^q_\rho : \rho\in 2^{n_q}\} \subset
   [\omega]^{<\aleph_0}$,
   \item $\Gamma_q\in [\omega_1]^{<\aleph_0}$,
   \item for $\alpha\neq\beta\in \Gamma_q$ and $\rho\in 2^{n_q}$,
       the union $(f_\alpha\restriction s^q_\rho) \cup 
       (f_\beta\restriction s^q_\rho)$ is not a function,
  \item for each $\rho\in 2^{n_q}$, the sequence
     $\{ s^q_{\rho\restriction j} : j\leq n_q\}$ is increasing,
     \item for $\rho,\psi\in 2^{n_q}$, 
        $s^q_\rho\cap s^q_\psi = 
          s^q_{\rho\restriction j} $ where 
          $j$ is maximal such that $\rho\restriction j = \psi\restriction j$.
  \end{enumerate}
  The ordering on $Q(\mathcal H)$ 
  is that $q\leq r$ providing 
   $n_q\geq n_r$, $\Gamma_q\supset \Gamma_r$,
    and $s^q_\rho=s^r_\rho$ for all 
     $\rho\in 2^{n_r}$.

     \medskip
     
     Fix any bijection $\varphi : \omega\times\omega \rightarrow\omega$.
     If $\{h_\alpha : \alpha\in\omega_1\}$ is coherent subfamily of 
     an $\omega^\omega$-family of functions, then
     let $Q_\varphi(\{h_\alpha : \alpha\in\omega_1\})$ 
     denote the poset $Q(\{ h_\alpha\circ \varphi : \alpha <\omega_1\})$.
\end{definition}

The idea for the poset $Q(\mathcal H)$
can be found in \cite{Farah2000}*{proof of 3.8.1}

\begin{definition}  Let $\mbox{ADF}_{\omega_1}$ denote 
the poset of finite partial functions 
 $p$ from $\omega_1\times\omega$ to $2$,
 such that, for some $n_p\in\omega$ 
 $\dom(p) = F_p\times n_p$. 
 
 Since the intention
 is to add a canonical almost disjoint family, the
 ordering on $\mbox{ADF}_{\omega_1}$ is that 
 $p<q$ providing $n_p\geq n_q$, $F_P\supset F_q$,
 and 
 for all $\alpha\neq \beta\in F_q$ and $n_q\leq j<n_p$, 
   $p(\alpha,j)\cdot p(\beta,j) = 0$ 
    (i.e. at least one has value $0$).
     
     For each $\iota\in\omega_1$, we let
      $\dot a(\iota)$ be the $\mbox{ADF}_{\omega_1}$
      name for the set $\{ j\in \omega : (\exists p\in G)~~p(\iota,j)=1\}$.
\end{definition}

\begin{lemma} Assume that $\mathcal H = \{h_\alpha : \alpha\in\omega_1\}$
 is a coherent family of partial functions on $\omega$ such 
 that $\{ \dom(h_\alpha) : \alpha \in \omega_1\}$ is mod finite increasing,
 and there is no $h\in\omega^\omega$ such
 that $\mathcal H\cup \{h\}$ is coherent. Then\label{luzinseq}
 the poset
  $Q(\mathcal H)$ is ccc.  Furthermore,  
  there is a condition $q\in   Q(\mathcal H)$ that
  forces  there is an uncountable $\Gamma\subset \omega_1$
  and an uncountable almost disjoint family $\mathcal A\subset
   [\omega]^{\aleph_0}$, such that, for each $a\in \mathcal A$,
     the sequence $\{ h_\alpha\restriction a : \alpha\in \Gamma\}$
     is a Luzin incoherent family.
\end{lemma}

\begin{proof} Let $\{ q_\xi : \xi\in\omega_1\}\subset Q(\mathcal H)$. 
By passing to an uncountable subcollection, we can assume
that there is a single sequence $\{ s_\rho : \rho\in 2^{\leq n}\}$
such that, for all $\xi\in\omega_1$ and $\rho\in 2^{\leq n}$,
 $n_{q_\xi}=n$ and $s^{q_\xi}_\rho = s_\rho$. 
For each $\xi\in\omega_1$, let $\Gamma_\xi = \Gamma_{q_\xi}$. 
By again passing to an uncountable subcollection we may
assume that the family $\{ \Gamma_\xi : \xi\in\omega_1\}$
is a $\Delta$-system with root $\Gamma'$.  For
each $\xi\in\omega_1$, let $\alpha_\xi$ be the minimum element
of $\Gamma_\xi\setminus \Gamma'$. 
 With yet another
such reduction, we may assume that there is an integer $m$
so that, for all $\xi$, $h_{\alpha_\xi}\restriction
 (\dom(h_{\alpha_\xi})\setminus m$ is a subset of
  $h_\beta$ for all $\beta\in \Gamma_\xi\setminus \Gamma'$.

Fix any countable elementary submodel $M$ of $H(\mathfrak c^+)$
with  the family $\{ h_{\alpha_\xi} : \xi\in\omega_1\}$
as an element. Let $\delta = M\cap \omega_1$.
Fix any $\xi\in\omega_1\setminus \delta$
and finite set $I\subset \dom(h_{\alpha_\xi})$. 
By simple elementary we have that 
the set $\{\eta \in\omega_1: h_{\alpha_\eta}\restriction I 
 = h_{\alpha_\xi}\restriction I\}$ 
is an element of $M$ and
is uncountable. Since the family $\mathcal H$
is mod finite increasing and is not coherently extendable,
 it follows that $\bigcup \{ h_{\alpha_\xi} \restriction
  \dom(h_{\alpha_\xi})\setminus m : \delta\leq \xi\}$
  is not a function. Choose any $i_0 >m$ so
  that there are $\delta\leq \xi_0,\eta_0$ with
   $h_{\alpha_{\xi_0}}(i_0)=j_0\neq k_0
    = h_{\alpha_{\eta_0}}(i_0)$.
    Let $J_0 = \{ \xi : h_{\alpha_\xi}(i_0)=j_0\}$
    and $K_0 = \{ \eta : h_{\alpha_\eta}(i_0)=k_0\}$. 
    Again, for every $i\geq i_0$,
    $\bigcup\{ h_{\alpha_\xi} \restriction 
    \dom(h_{\alpha_\xi})\setminus i:\xi\in J_0\setminus\delta
    \}$
    is not a function but the domain of this relation will 
    mod finite contain the domain of
    $\bigcup\{ h_{\alpha_\eta} \restriction 
    \dom(h_{\alpha_\xi})\setminus i:\eta\in K_0\setminus\delta
    \}$. Therefore we may choose an $i_1 >i_0$ so that there 
    are $\xi_1\in J_0$ and $\eta_1\in K_0$ so that
       $h_{\alpha_{\xi_1}}(i_1)=j_1\neq k_1 = h_{\alpha_{\eta_1}}(i_1)$.
       It should be clear that continuing with such a recursive construction 
       we can find a sequence $\{ i_\ell : \ell < 2^{n+1}\} $ so that 
       there is a pair $\xi,\eta$ satisfying that
         $h_{\alpha_\xi}(i_\ell)\neq h_{\alpha_\eta}(i_\ell)$ for all $\ell <2^{n+1}$.
         
         We are ready to find a condition $r$ that is below each
         of $q_\xi$ and $q_\eta$. We let $n_r = n+1$
         and we re-index $\{ i_\ell : \ell < 2^{n+1}\}$
           as $\{ i_\rho : \rho\in 2^{n+1}\}$.
         For each $\rho\in 2^{n+1}$, 
            the definition of $s^r_\rho$ is
             $s_{\rho\restriction n}\cup \{i_\rho\}$. 
             For $\rho\in 2^{\leq n}$, $s^r_\rho = s_\rho$.
    Clearly, for each $\rho\in 2^{n+1}$, 
    we have that $(h_{\alpha_\xi}\restriction s^r_\rho)
     \cup (h_{\alpha_\eta}\restriction s^r_\rho)$ is not a function
     because $h_{\alpha_\xi}(i_\rho) \neq h_{\alpha_\eta}(i_\rho)$. 
     By the choice of $m$ it also therefore follows that
     for all $\beta_0\in \Gamma_{\xi}\setminus\Gamma'$
     and $\beta_1 \in \Gamma_{\eta}\setminus \Gamma'$,
     $(h_{\beta_0}\restriction s^r_\rho)
     \cup (h_{\beta_1}\restriction s^r_\rho)$ is not a function.
             If $\beta_0\neq \beta_1$ are both elements
             of $\Gamma_\xi$ or elements of $\Gamma_\eta$,
             then we already have that
     $(h_{\beta_0}\restriction s^r_{\rho\restriction n})
     \cup (h_{\beta_1}\restriction s^r_{\rho\restriction n})$
     is not a function.
             
 Since $Q(\mathcal H)$ is ccc, there is a condition
  $q$ that forces that the generic filter $G$ is uncountable. 
  Consider any such generic filter $G$.
  It is easily checked that,
  for each $n\in\omega$, $D_n = \{ r\in Q(\mathcal H) : 
    n_r \geq n\}$ is dense. 
    Let $\Gamma = \bigcup \{ \Gamma_r : r\in G\}$
    and, for each $\rho\in 2^\omega$,
     let $a_\rho = \bigcup \{ a^r_{\rho\restriction n_r} : r\in G\}$.
     The family $\mathcal A = \{ a_r : r\in 2^\omega\}$
 is almost disjoint. We check that for  $r\in 2^\omega$,
  the family $\{  h_\beta \restriction a_r : \beta\in \Gamma\}$
  is Luzin incoherent. Let $\alpha\neq \beta\in \Gamma$ and
  choose $r \in G$ so that $\{\alpha,\beta\}\subset \Gamma_r$. 
  Since $(h_{\alpha}\restriction s^r_{\rho\restriction n})
     \cup (h_{\beta }\restriction s^r_{\rho\restriction n})$
     is not a function, 
     it of course follows that 
     $(h_{\alpha}\restriction a_r )
     \cup (h_{\beta }\restriction a_r  )$
     is not a function. 
\end{proof}

For a regular cardinal $\kappa>\omega_1$, $S^\kappa_1$ 
denotes the set of $\lambda<\kappa$ that have cofinality
$\omega_1$.  
For the remainder of the paper we
assume that   there
is a $\diamondsuit(S^\kappa_1)$-sequence
$\{ A_\lambda : \lambda\in S^\kappa_1\}$.
Of course what this means is that
each $A_\lambda\subset \lambda$ and for
all $A\subset \kappa$, the
set $\{\lambda\in S^\kappa_1 : A\cap \lambda=A_\lambda\}$
is stationary. 
This however is not the sequence that we will
call our $\diamondsuit(S^\kappa_1)$-sequence.
 We also assume that  GCH holds below $\kappa$
 and so $H(\kappa)$ has cardinality $\kappa$.

\begin{definition}
Fix any enumeration $\{ d_\xi : \xi \in \kappa\}$
of $H(\kappa)$. 
For each $\lambda\in S^\kappa_1$,
 let $D_\lambda = \{ d_\xi : \xi \in A_\lambda\}$.
 We refer to
 $\{ D_\lambda : \lambda\in S^\kappa_1\}$
 as our $\diamondsuit(S^\kappa_1)$-sequence.
\end{definition}

\begin{lemma}
Let  $\langle P_\alpha, \dot Q_\beta : \alpha \leq \kappa,
  \beta < \kappa\rangle$ denote   a finite support ccc iteration
  of posets of cardinality less than $\kappa$.
  In addition, make the following assumptions\label{iterseq1}
  about this sequence.
  \begin{enumerate}
  \item For each $\omega_1 \leq \lambda<\kappa$ and each $P_\lambda$-name
  of a $\sigma$-centered poset $\dot Q$, there is a $\beta<\lambda
  +\lambda$ such that $\dot Q_\beta = \dot Q$,
  \item for each $\lambda\in S^\kappa_1$, if $D_\lambda$ is a
  $P_\lambda$-name that is forced, by $1$,
  to be a maximal coherent 
  set
   $\mathcal H$ of functions whose domains form a ccc over
   fin $P$-ideal,
    then, if  there is a sequence $\mathcal H_\lambda =
    \{ \dot h_\alpha : \alpha\in \omega_1\}$
    such that $1\Vdash  \mathcal H_\lambda \subset \mathcal H$,
     $1$ forces that 
      $\mathcal H_\lambda$ is non-extendable 
      and the domains are mod finite increasing,
      then $\dot Q_{\lambda+1} = Q(\mathcal H_\lambda)$,
      \item for each $\lambda\in S^\kappa_1$, 
      if $D_\lambda$ is a $P_\lambda$-name of a
      non-extendable
      coherent 
       $\omega^\omega$-family, $\mathcal H $,
       then (since $P_\lambda\Vdash \mathfrak b = \aleph_1$),
        we let $\dot Q_\lambda = Q_\varphi(\mathcal H_\lambda)$
        where $\mathcal H_\lambda\subset \mathcal H$
        is a cofinal subsequence of order type $\omega_1$.        
  \end{enumerate}
  Then $P_\kappa$ forces that the principles $\omega^\omega$-cohere
  and P-cohere hold.
\end{lemma}

\begin{proof} 
We begin with $\omega^\omega$-cohere. It is clear
that $P_{\kappa}$ forces that $\mathfrak d = \mathfrak b = \kappa$.
Fix a sequence $\{ \dot f_\xi : \xi < \kappa\}$ 
of (countable) $P_\kappa$-names of elements of $\omega^\omega$
such that the sequence is forced to be an $<^*$-increasing
and $<$-cofinal subset of $\omega^\omega$. To prove
that $\omega^\omega$-cohere holds it suffices to consider
a sequence  $\{ \dot h_\xi : \xi < \kappa\}$ of 
countable $P_\kappa$-names such that $1$ forces
that $ \dot h_\xi$ is a function into $\omega$
with domain equal to $\dot f_\xi^\downarrow$. 
 Fix any  sequence $\{M_\mu : \mu<\kappa\}$
 of elementary submodels of $H(\kappa^+)$ such 
 that
  $\{ \dot h_\xi : \xi < \kappa\}\in M_0$, 
   each $M_\mu$ has cardinality less than $\kappa$,
   $M_\mu^\omega \subset M_{\mu+1}$ for all $\mu$,
   and $M_\lambda = \bigcup \{ M_\mu : \mu <\lambda\}$
   for all limit $\lambda$. 
 Let $C= \{ \mu < \kappa : M_\mu\cap \kappa= \mu \}$. 
 It is well-known that $C$ is a cub subset of $\kappa$.
 Choose
 any $\lambda\in C $ with uncountable cofinality. Then
 it follows that $\{ \dot f_\xi : \xi < \lambda\}
 \cup \{ \dot h_\xi : \xi < \lambda\}$ are $P_\lambda$-names,
 and that $1 \Vdash_{P_\lambda} $ ``$\{ \dot f_\xi : \xi < \lambda\}$
 is cofinal in $\omega^\omega$''. In addition, by elementarity,
  if there is $P_\lambda$-name $\dot h \in \omega^\omega$ that
  is forced to 
  coherently extend  the sequence $\{ \dot h_\xi : \xi<\lambda\}$,
  then there is such a name in $M_\lambda$ which will 
  then be forced  by $P_\kappa$ to coherently
  extend the entire sequence $\{ \dot h_\xi : \xi <\kappa\}$. 
  Now there must be such a   $\dot h$, since otherwise
    the poset
    $Q_{\lambda+1}$ renders the initial segment
     $\{ \dot h_\xi : \xi < \lambda\}$ non-extendable
      -- contradicting the existence of $\dot h_{\lambda+1}$. 
      \medskip
      
      Now assume that $\{ \dot h_\xi : \xi < \kappa\}$ 
      is a sequence of countable $P_\kappa$-names
      that are forced to be a coherent family of partial
      functions on $\omega$ and that the family
       $\{ \dom(\dot h_\xi ) : \xi  < \kappa\}$
       is forced to generate an ideal $\dot {\mathcal I}$
       that is
  ccc over fin $P$-ideal. 
      By enriching the list of names,
      there is no loss to assume that for
       each countable $S\subset\kappa$, there is
        an $\eta < \kappa$ satisfying that
         $1\Vdash \dom (\dot h_\xi) \subset^* \dom(\dot h_\eta)$
         for all $\xi\in S$.
   Again choose a 
   sequence $\{M_\mu : \mu<\kappa\}$
 of elementary submodels of $H(\kappa^+)$ such
 that each $M_\mu$ has cardinality less than $\kappa$,
 $\{\dot h_\xi : \xi\in\kappa\}\in M_0$, 
   $M_{\mu+1}^\omega \subset M_{\mu+1}$ for all $\mu$,
   and $M_\lambda = \bigcup \{ M_\mu : \mu <\lambda\}$
   for all limit $\lambda$. 
 Let $C= \{ \mu < \kappa : M_\mu\cap \kappa= \mu \}$. 
 Again choose any $\lambda\in C\cap S^\kappa_1$ 
 so that $D_\lambda$ is the sequence
  $\{ \dot h_\xi : \xi<\lambda\}$. 
   Choose any sequence $\{\mu_\beta :\beta < \omega_1\}
   \subset \lambda$ 
   of successor ordinals that is cofinal. 
   Recall that $M_{\mu_\beta}^\omega \subset
   M_{\mu_\beta}$ for all $\beta < \omega_1$. 
   For each $\beta<\omega_1$, choose 
    $\xi_\beta \in M_{\mu_\beta}\setminus
     \bigcup\{M_{\mu_\zeta}:\zeta <\beta\}$ so
     that $\dot Q_{\xi_\beta}$ is the 
      $P_{\xi_\beta}$-name of the poset
       $\mbox{ADF}_{\omega_1}$. 
       For each $\beta$, 
       let $\{ \dot a(\xi_\beta,\iota) : \iota< \omega_1\}$
       denote the generic almost disjoint family 
       added by $\mbox{ADF}_{\omega_1}$. 
       For
       each $\beta<\omega_1$,
         fix a choice $\iota_\beta<\omega_1$ such
         that $1\Vdash_{P_\kappa} \dot a(\xi_\beta,\iota_\beta)
          \in \dot {\mathcal{I}}$.  
          We note that, by elementarity,
            $h_{\dot a(\xi_\beta,\iota_\beta)}$
            has a $P_{\mu_\beta}$-name. 
          
          Now let $G_\lambda$ be a $P_\lambda$-generic
          filter and we finish the argument in $V[G_\lambda]$.
By recursion on $\gamma<\omega_1$, we can choose
  $a_\gamma\in \val_{G_\lambda}(\dot {\mathcal I})$  so that,
  for all $\beta < \gamma$, each of $a_\beta$
and    $\val_{G_\lambda}(\dot a(\xi_\beta,\iota_\beta))$
are mod finite contained in $a_\gamma$. For
each $\gamma<\omega_1$, let $\bar h_\gamma$ denote the usual
 $h_{a_\gamma}$. We obtain a contradiction by assuming
 that, in $V[G_\lambda]$, no $h\in \omega^\omega$ satisfies
 that $\{ \val_{G_\lambda}(\dot h_\xi) : \xi < \lambda\}
 \cup \{h\}$ is coherent.
 Consider any $h\in \omega^\omega\cap 
   V[G_\lambda]$. Choose $\beta < \omega_1$ so that
    $h\in V[G_{\mu_\zeta}]$ for some $\zeta<\beta$,
    and fix such a $\zeta$. 
    Let $p\in G_{\mu_\beta}$ be arbitrary
    and let   $m<\omega$ be arbitrary.  
Choose  $\xi < \mu_{\zeta}$  such that, in $V[G_{\mu_\zeta}]$,
  $y = \{ n \in \dom(\val_{G_{\mu_\zeta}}(\dot h_\xi)):
    h(n)\neq \val_{G_{\mu_\zeta}}(\dot h_\xi)(n) \}$
    is infinite.  If we jump to 
      $V[G_{\xi_\beta}]$ it is clear that
       the set of $r\in 
       \mbox{ADF}_{\omega_1}$ such that
       $r<p(\xi_\beta)$ and $r((\iota_\beta,j)) = 1$
       for some $j\in y\setminus m$, is pre-dense
       below $p(\xi_\beta)$. This proves 
       that $\{ h,\bar h_{\beta+1}\}$ is not coherent.
Therefore we have proven that the desired sequence
$\mathcal H_\lambda = \{ \bar h_\alpha : \alpha < \omega_1\}$ for defining
 $\dot Q_\lambda = Q(\mathcal H_\lambda )$
 exists
in $V[G_\lambda]$. 
This completes the proof since,
 by Lemma \ref{luzinseq}, 
   there is clearly
     an uncountable almost disjoint family of sets
     that are forced by $P_{\lambda+1}$ to not be elements of $\dot{\mathcal I}$.
 \end{proof}
 
 \section{Cohen reals and Hausdorff gaps}
 
 In this section we discuss Shelah's method of  
 constructing an $(\omega_1,\omega_1)$-gap with the property that one
 can generically split the gap while ensuring that
 the image of the gap by a non-trivial automorphism 
  is not split in the forcing extension. 
  
 It will be convenient to adopt the following non-standard
 presentation of a   gap structure. 
 
 \begin{definition} For an ordinal $\delta\leq\omega_1$,
  a $\delta$-gap is a sequence $\vec{a} = \langle a_\alpha, x_\alpha : \alpha  
   < \delta\rangle$ such that, for all $\alpha \leq \beta < \delta$:
   \begin{enumerate}
   \item  $a_\alpha\subset^* a_\beta\subset\omega$,
   \item $x_\alpha \subset a_\alpha$,
   \item $x_\beta\cap a_\alpha =^* x_\beta$.
   \end{enumerate}
   A set $Y$ splits the $\delta$-gap $ \langle a_\alpha, x_\alpha : \alpha  
   < \delta\rangle$, if $Y\cap a_\alpha =^* x_\alpha$ for all $\alpha<\delta$.
   We say that a gap can not be split if there is no such $Y$.
 \end{definition}

 \begin{definition}
 A Hausdorff gap
 is an $\omega_1$-gap $\langle a_\alpha, x_\alpha :
 \alpha <\omega_1\rangle$ that has the property
 that for each $\alpha\in\omega_1$ and
  $n\in\omega$,
   the set $\{\beta <\alpha :
    x_\alpha\cap (a_\beta\setminus x_\beta)\subset 
     n\}$ is finite.
 \end{definition}
 
 \begin{proposition}[\cite{Kunen}*{II Exercise (24)}]
 A Hausdorff $\omega_1$-gap is not split.
 \end{proposition}
 
 There are two natural things one may be interested
 in doing when presented with an
  $\omega_1$-gap: \textit{splitting\/}
  the gap or \textit{freezing\/} an unsplit 
  gap.  
  It is a well-known unpublished result of Kunen
  that if an $\omega_1$-gap 
  is not split, then there is a ccc poset 
  which will freeze the gap in the sense
  that it has a Hausdorff subgap.
  
  \begin{proposition}[\cite{Baumgartner}*{4.2}]
  If $\vec a = \langle a_\alpha,x_\alpha :\alpha\in \omega_1\rangle$
  is\label{freeze}
  an $\omega_1$-gap that is not split,
   then there is a  ccc poset $Q(\vec a)$
   that introduces an increasing function 
     $f\in {}^{\omega_1}\omega_1$ such that
      $\langle a_{f(\alpha)}, 
       x_{f(\alpha)} : \alpha\in\omega_1\}$
       is a Hausdorff gap.
  \end{proposition}
  
  Splitting an $\omega_1$-gap with a ccc poset
  can not be done in general (e.g. a Haudorff
  gap simply can not be split).
 
 \begin{definition} For a $\delta$-gap $\vec a
 =\langle a_\alpha, x_\alpha : \alpha < \delta\rangle$, 
 we define the poset $P_{\vec a} = 
 P_{\langle a_\alpha,x_\alpha\rangle_{\alpha<\delta} }$
 (or $P_{\langle a_\alpha,x_\alpha : {\alpha<\delta} \rangle}$)
 to be the set of conditions $p = (x_p, n_p, L_p)\in \mathcal P(\omega)
 \times \omega \times [\delta]^{<\aleph_0}$,
 where:
 \begin{enumerate}
 \item $x_p\subset n_p\cup \bigcup\{ a_\alpha :\alpha \in L_p\}$,
 \item $a_\alpha\setminus a_\beta\subset n_p$ for all $\alpha\leq \beta $ both  in $L_p$,
  \item $x_\alpha\Delta (a_\alpha\cap x_p) \subset n_p$ for each $\alpha\in L_p$.
 \end{enumerate}
 The ordering on $P_{\vec a}$ is given by $p<q$ providing
 \begin{enumerate}
 \setcounter{enumi}{3}
    \item $x_q\subset x_p$, $n_q \leq n_p$, $L_q\subset L_p$,
    \item $x_p\cap n_q = x_q$,
    \item $x_p\cap a_\alpha = x_q\cap a_\alpha$ for all $\alpha\in L_q$.
 \end{enumerate}
 \end{definition}
 
 \begin{proposition} For any $\delta$-gap $\vec a$ and any $\beta < \delta$,
   the poset $P_{\vec a\restriction \beta}$ is a subposet
   of $P_{\vec a}$.  The canonical $P_{\vec a}$-name $\dot Y_{\vec a}$
   given by $\{ (k,p) : p\in P_{\vec a}\ \ , \ \ k\in x_p\}$ satisfies
   that $1$ forces that $\dot Y_{\vec a}$ splits the gap $\vec a$.
 \end{proposition}
 
 The poset $P_{\emptyset}$ can be identified with a standard
 poset $\mathcal C_\omega\subset [\omega]^{<\aleph_0}\times \omega$
 for adding a Cohen subset of $
 \omega$. A condition $(s,n)\in \mathcal C_\omega$ providing
  $s\subset n\in \omega$ and $(s,n)\leq (t,m)$ providing
    $m\leq n$ and $s\cap m = t$. 
    The canonical name, $\dot c$, for the Cohen real added
    is the set $\{ (j, (s,n) )  :  j\in s\}$. 
    Another standard poset we mention here is the dominating
    real poset $\mathbb D$. Just to be specific we define
    it in such a way that the generic dominating function
    added is strictly increasing.
    
    \begin{definition} The poset $\mathbb D$ is  the
    poset consisting of pairs $(\sigma, f)$ where $\sigma
    \in \bigcup_n {}^n\omega$ is a strictly increasing function
    and $f\in \omega^\omega$. A condition $(\sigma, f)$
    extends $(\tau,h)$ providing $\tau\subset \sigma$, 
      $h\leq f$, and $\sigma(k)>f(k)$ for all $k\in \dom(\sigma)\setminus
      \dom(\tau)$.
    \end{definition}
 
 For each $\omega\leq \delta\in\omega_1$, fix, in the ground model, 
 an enumeration
 $\{ \alpha^\delta_i : i\in\omega\}$
 of $\delta$.

 \begin{proposition} Let    $\vec a = 
 \langle a_\alpha,x_\alpha 
 :\alpha <\delta\rangle $ ($\delta \leq\omega_1$) be any
 $\delta$-gap
 and let\label{finitechanges}  
 $p\in P_{\vec a}$. Then for every $s\subset n_p$ and
 every $m\geq n_p$, each of $p^\downarrow s$ and $p^\uparrow m$
 are in $P_{\vec a}$ where
 \begin{enumerate}
 \item $p^\downarrow s = ( s\cup (x_p\setminus n_p), n_p, L_p)$, and
 \item $p^\uparrow m = (x_p , m, L_p)$.
 \end{enumerate}
 In addition $p^\uparrow m\leq p$.
 \end{proposition}
 
 \begin{proof} The verification that $p^\uparrow m\in P_{\vec a}$
 and $p^\uparrow m\leq p$ is completely routine so we skip
 the details. Now let $r=p^\downarrow s$ so that we can
 refer easily to $x_r$. Clearly $x_r\setminus n_r = x_r\setminus n_p
  = x_p\setminus n_p$. Each of the conditions (1), (2), and (3)
  in the definition of $P_{\vec a}$ really only depend on
  the value of $x_p\setminus n_p$ and so the condition
   $r$ also satisfies those conditions.
 \end{proof}

\begin{definition}  A $\delta$-gap,
   $\langle a_\alpha , x_\alpha : \alpha < \delta \rangle$,
    is a $W$-sealing extension of  $\langle   a_\alpha ,x_\alpha :
    \alpha < \beta\rangle$ providing\label{definesealing}
    \begin{enumerate}
    \item $W$ is a transitive submodel of a sufficient fragment of ZF, 
    \item $\beta   < \delta \leq\omega_1$,
    \item $\{ a_\alpha , x_\alpha : \alpha < \beta\}\in W $,
    \item there is an infinite set
     $J \subset\omega$ such that
    for every $\omega$-sequence $\mathcal D =
    \langle D_n:n\in\omega\rangle$
    in $W$ consisting of dense subsets of   $  
    P_{\{ a_\alpha ,x_\alpha :\alpha <\beta\}}$,
     there is an $m_{\mathcal D}$ such
    that for all
       $p\in P_{\langle a_\alpha,x_\alpha :\alpha < \beta\rangle}$, 
        with   $m_{\mathcal D}\leq n_p \in J$ 
        and $L_p\subset \{\alpha^\delta_i : i< n_p\}$,
        there is a  $d<p$ such that $d\in \bigcap\{D_k : k<n_p\}$, 
  $x_d\cap  n_d\setminus n_p =  x_\delta\cap
  n_d\setminus n_p$, $n_d\setminus n_p\subset a_\delta$,
  and,
        for all $\alpha\in L_d$,         
        $a_\alpha\setminus a_\delta \subset n_d$ and
         $x_\delta\cap a_\alpha\setminus n_d = x_\alpha\setminus  
         n_d$. 
    \end{enumerate}
\end{definition}

\begin{proposition} For $\beta < \delta<\omega_1$, a
  $\delta$-gap $\vec a$ is a $W$-sealing of
 $\vec a\restriction \beta$ if and only if $\vec a\restriction (\beta{+}1)$
 is a $W$-sealing of $\vec a\restriction\beta$.
\end{proposition}

\begin{lemma} Let $\langle P_\xi, \dot Q_\zeta : 
\xi\leq\omega_1, \zeta <\omega_1\rangle$
 be  a finite support iteration of ccc posets and let $G$ be a $P_{\omega_1}$-generic 
 filter. Assume that $\vec a = \langle a_\alpha , x_\alpha : \alpha \in \omega_1\rangle$ 
 is an $\omega_1$-gap in $V[G]$. 
 For $\delta<\omega_1$, let $G_\delta = G\cap P_\delta$.
 Then $P_{\vec a}$ is ccc providing 
 there is a stationary set $S\subset \omega_1$ satisfying that, 
 for all $\delta\in S$,  $\vec a\restriction \delta 
   = \langle a_\alpha, x_\alpha : \alpha < \delta\rangle$
 is an element of $V[G_\delta]$ and\label{cccsplitgap} 
    $\vec a\restriction \delta{+}1$ 
    is a  $V[G_\delta]$-sealing extension of $\vec a\restriction \delta$.
\end{lemma}

\begin{proof} Let $D$ be any dense subset of $P_{\vec a}$,
and fix a canonical $P_{\omega_1}$-name $\dot D$ for $D$.
For each $\beta < \omega_1$,  $P_{\vec a\restriction\beta}$ is a 
countable poset and $P_{\vec a} = \bigcup \{ P_{\vec a\restriction\beta} : \beta
<\omega_1\}$. Therefore there is a cub $C\subset \omega_1$ (in $V[G]$)
 satisfying that for all $\gamma_0 < \gamma_1$ from $C$, 
 the name $\dot D\cap P_{\vec a\restriction \gamma_0}$ is a 
 $  P_{ \gamma_1}$-name, 
 and 
 for each $p\in P_{\vec a\restriction\gamma_0}$, there is a 
  $d\in D\cap P_{\vec a\restriction \gamma_1}$ with $d \leq p$.
  Choose any cub $C_1$ from $V$ such that $C_1\subset C'$. 
  Fix any $\delta\in C_1\cap S$. We check that $D\cap 
    P_{\vec a\restriction \delta}$ is pre-dense in $P_{\vec a}$. 
    Firstly, since $\delta\in C'$, it follows that
     $P_{\vec a\restriction\delta} =\bigcup \{ P_{\vec a\restriction
      \gamma} : \gamma\in C\cap \delta\}$. This implies
      that $D\cap P_{\vec a\restriction \delta}$ is a dense
      subset of $P_{\vec a\restriction\delta}$. 
      Similarly, it is evident that the name for
       $\dot D\cap P_{\vec a\restriction \delta}$ 
       is a $P_\delta$-name. By the assumption on $S$,
         $\vec a\restriction (\delta{+}1)$ is 
           a $V[G_\delta]$-sealing extension of $\vec a\restriction
           \delta$. Let $J$ be the infinite set as in the
           definition of     $\vec a\restriction(\delta{+}1)$ 
           being   sealing and let $m_{D,\delta}$ be the 
           corresponding lower
           bound for the dense set $D\cap P_{\vec a\restriction \delta}$.
Now consider any $p\in P_{\vec a}$ 
and we wish to find $d\in D\cap 
 P_{\vec a\restriction\delta}$ that is compatible with $p$. 
 By extending $p$ we may assume that $\delta\in L_p$,
  $m_{D,\delta}\leq n_p\in J$, 
 and
 that $\emptyset\neq L_p\cap \delta \subset \{\alpha^\delta_i : i
 <n_p\}$.
 Let $\beta = \max(L_p\cap \delta)$.
 Set $q = ( x_p\cap (n_p\cup a_\beta), n_p, L_p\cap \delta)$. 
 It is routine to check that $q\in P_{\vec a\restriction\delta}$
 and that $x_q \setminus n_q = x_\beta\setminus n_q$. 
 Using the assumptions on $m_{D,\delta}$ and $J$, 
  we may choose $d\in D\cap P_{\vec a\restriction \delta}$ so
  that $d\leq q$,  $x_d\cap n_p = x_p\cap n_p$, 
  $x_d\cap (n_d\setminus n_p)  =  x_\delta\cap
  (n_d\setminus n_p)$, and
  for all $\alpha \in L_d$,  $a_\alpha\setminus n_d\subset a_\delta $
  and 
  $x_\delta\cap a_\alpha\setminus n_d = x_\alpha\setminus n_d$.
  Note also that $n_d\setminus n_p\subset a_\delta$
  and this implies that $x_\gamma\cap (n_d\setminus n_p) 
    = x_\delta\cap (n_d\setminus n_p)$ for all $\gamma\in L_p\setminus
    \delta$.  This then implies that $x_p\cap (n_d\setminus n_p)
    $ is equal to $x_\delta\cap (n_d\setminus n_p)$,
     which, in turn, is equal to $x_d\cap (n_d\setminus n_p)$.
  
  Define $r=( x_p, n_d, L_d\cup L_p)$. 
  We complete the proof by checking
  that $r\in P_{\vec a}$, $r\leq p$ and $r\leq d$. 
  Since $p\in P_{\vec a}$, 
  to show that $r\in P_{\vec a}$, 
  it suffices to
  only consider
   $\alpha\in L_d\setminus L_p$, $\gamma\in L_p\setminus \beta$
   and show that $x_p\cap (a_\alpha\setminus n_d)= x_\alpha\setminus n_d$
   and that $a_\alpha\setminus a_\gamma\subset n_d$. 
   By assumption $a_\alpha\setminus a_\delta \subset n_d$  
   and $a_\delta\setminus a_\gamma\subset n_p\subset n_d$. 
   Therefore we do have that $a_\alpha\setminus a_\delta\subset n_d$.
   Similarly, $x_p\cap (a_\delta\setminus n_p) = x_\delta\setminus n_p$
   and $x_\delta\cap (a_\alpha\setminus n_d) = x_\alpha\setminus n_d$.
   Therefore $x_p\cap (a_\alpha\setminus n_d) = x_\alpha\setminus n_d$
   as required. The fact that $r\leq p$ is trivial since $x_r = x_p$.
   To check that $r\leq d$ it remains to   show that
    $x_r\cap a_\alpha = x_d\cap a_\alpha$ for all $\alpha\in L_d$. 
    In fact, since $n_r = n_d$, it suffices to show that
      $x_r\cap a_\alpha\cap n_d = x_d\cap a_\alpha\cap n_d$, and
      this holds since $x_p\cap n_d =
      (x_p\cap n_p)\cup (x_p\cap (n_d\setminus n_p))
       = (x_d\cap n_p)\cup (x_d\cap (n_d\setminus n_p))$. 
\end{proof}

Now we strengthen the previous lemma and incorporate a
name of a non-trivial automorphism.

\begin{lemma} Let $\langle P_\xi, \dot Q_\zeta : \xi\leq\omega_1, \zeta <\omega_1\rangle$
 be  a finite support iteration of ccc posets and let $\dot F$ be a
  $P_{\omega_1}$-name that is forced to induce an automorphism
  on $\mathcal P(\omega)/\mbox{fin}$. Assume also that, for all
  successor $\xi\in\omega_1$ and $P_\xi$-name $\dot x$ of a subset of $\omega$,
    $\dot F(\dot x)$ is a $P_\xi$-name. 
 
 Let $G$ be a $P_{\omega_1}$-generic 
 filter  and  $\vec a = \langle a_\alpha , x_\alpha : \alpha \in \omega_1\rangle$ 
 an $\omega_1$-gap in $V[G]$ such that, for all limit $\delta<\omega_1$
 and $n\in\omega$, 
     $\vec a\restriction (\delta +n)$ 
 is an element of $V[G_{\delta+n}]$ and\label{cccfreezegap} 
    $\vec a\restriction (\delta{+}2n{+}1)$ 
    is a  $V[G_{\delta{+}2n}]$-sealing extension of 
    $\vec a\restriction (\delta{+}2n)$. 
    
    Finally, assume that for all limit $\delta<\omega_1$ and
     $n\in \omega$, the following forcing statement holds,
     in $V[G_{\delta+2n+2}]$, 
     for all $P_{\vec a\restriction (\delta+2n)}$-names
      $\dot Y$ in $V[G_{\delta+2n}]$, 
      \[ 1 \Vdash_{P_{\vec a\restriction (\delta+2n+1)}}
      \dot Y\cap  F(a_{\delta+2n+1}) 
       \neq^* F(x_{\delta+2n+1})~.\]       
       Then, in the ccc forcing extension by $P_{\omega_1}*
        P_{\vec a}$, the $\omega_1$-gap
        given by 
          $\langle F(a_\alpha), F(x_\alpha) : \alpha < \omega_1\rangle$
    is not split.      
\end{lemma}

\begin{proof} Let $\dot Y\in V[G_{\omega_1}]$ be any
element of $\mathcal P(\omega,P_{\langle
 a_\alpha, x_\alpha : \alpha < \omega_1\rangle})$.
 Since $
P_{\langle
 a_\alpha, x_\alpha : \alpha < \omega_1\rangle}$
 is ccc,  $\dot Y$ is countable. 
 Therefore there is a limit $\delta < \omega_1$
 such that $\dot Y\in V[G_\delta]$ and
 such that $\dot Y \in \mathcal P(\omega,P_{
 \langle a_\alpha , x_\alpha : \alpha < \delta\rangle})$.
 By assumption, we then have, 
 in $V[G_{\delta+2}]$, 
 \[ 1 \Vdash_{P_{\vec a\restriction (\delta+1)}}
      \dot Y\cap  F(a_{\delta+1}) 
       \neq^* F(x_{\delta+1})~.\]       
For each $n\in\omega$, let
 $D_n\in V[G_{\delta+2n}]$ be the dense set of
 conditions $p\in 
 P_{\vec a\restriction (\delta+1)}$ that
 satisfy that, for some $n<j\in F(a_{\delta+1})$,
 \begin{enumerate}
 \item $j\notin F(x_{\delta+1})$ implies
 that 
  $p\Vdash_{P_{\vec a\restriction (\delta+1)}}
       j \in \dot Y$,
  \item $j\in F(x_{\delta+1})$ implies
  that $p\Vdash_{P_{\vec a\restriction (\delta+1)}}
         j\notin \dot Y$.
  \end{enumerate}
Then, since    $\langle
 a_\alpha, x_\alpha : \alpha < \omega_1\rangle$
 is a $V[G_{\delta+2}]$-sealing extension
 of $\langle
 a_\alpha, x_\alpha : \alpha < \delta+2n\rangle$,
 each $D_n$ remains a pre-dense
 subset of
 $P_{\langle a_\alpha , x_\alpha : \alpha <\omega_1\rangle}$.
 This proves that $\dot Y$ is forced to not split
 the gap 
          $\langle F(a_\alpha), F(x_\alpha) : \alpha < \omega_1\rangle$. 
\end{proof}

For easy reference, let us make a definition
of \textit{F-avoiding\/}.

\begin{definition}
For a function
 $F\in \mathcal P(\omega)^{\mathcal P(\omega)}$, a
$(\delta{+}2)$-gap,
  $\vec a = 
  \langle  a_\alpha,x_\alpha : \alpha <  \delta+2\rangle$, 
\textit{F-avoids\/} all $W$-names, if the forcing
statement
\[ 1 \Vdash_{P_{\vec a }}
      \dot Y\cap  F(a_{\delta+1}) 
       \neq^* F(x_{\delta+1})~.\]       
      holds  
  for all $P_{\vec a }$-names
      $\dot Y$  in $W$.
\end{definition}

Say that a $\delta$-gap $\langle a_\alpha, x_\alpha : 
  \alpha < \delta\rangle$ is an $(F,\delta)$-gap so long
  as $\omega\setminus a_\alpha\notin \mbox{triv}(F)$ 
  for all $\alpha <\delta$.
In the next section we   direct
      our efforts at producing such 
      $(F,\delta+2)$-gap extensions
      of an $(F,\delta)$-gap that 
      are sufficiently {F-avoiding}.

\section{Adding Cohen reals to F-avoid names}

In this section we examine the extra properties to place
on the iteration sequence $\langle P_\alpha, \dot Q_\beta :
 \alpha \leq \kappa, \beta < \kappa\rangle$ from Lemma \ref{iterseq1}
 so as to be able to produce  a sequence as in 
 Lemma \ref{cccfreezegap} if $\dot F$ is a $P_\kappa$-name
   that is   forced to induce a non-trivial automorphism. 
 At a basic level, if for $\lambda\in S^\kappa_1$, 
 the $\diamondsuit(S^\kappa_1)$-element
 $D_\lambda$ is a $P_\lambda$-name
 $\dot F$ that is forced to induce a non-trivial automorphism
 on $\mathcal P(\omega)/\mbox{fin}$, then we will,
 if possible, choose $\dot Q_\lambda$ to be 
  $P_{\vec a}$ for an $\omega_1$-gap as in Lemma \ref{iterseq1}. 
  We will review below that, just as in 
  Shelah-Steprans \cite{ShSt88}, there will then be a
   suitable $\dot Q_{\lambda+1}$ (see
    Proposition \ref{freeze})
   so that $P_{\lambda+2}$
   will then force that $\dot F$  does not extend 
   to all of $\mathcal P(\omega)/\mbox{fin}$ in the final model.
   That   is,
   we are following the method of Shelah-Steprans \cite{ShSt88}
    except that in the case when $\kappa=\aleph_2$ it is shown
    to be sufficient to, loosely speaking,
    at each step $\delta+2n+1$, 
    seal only countably
    many pre-dense subsets of
    $P_{\vec a\restriction (\delta{+}2n{+}1)}$  and to 
     \textit{F-avoid} only countably many 
       $P_{\vec a\restriction (\delta{+}2n{+}1)}$-names.
There are three basic minimal requirements that must be met
at each step of choosing $P_{\mu_\alpha}, P_{\mu_{\alpha+1}}$
and the
$P_{\mu_{\alpha+1}}$-names 
   $\dot a_\alpha, \dot x_\alpha$. 
The first is related to the fact that it is impossible
to \textit{F-avoid} a name $\dot Y$ if $F$ is an  
automorphism induced by $h\in \omega^\omega$
and $\dot Y$ is simply the name
with the property that every $p\in P_{\vec a}$
forces that $\dot Y\cap F(a_\alpha) = h(x_\alpha)$
for all $\alpha\in \dom(\vec a)$. 
Therefore,   we are always having to ensure
that $\omega\setminus a_\alpha\notin \mbox{triv}(F)$,
and, additionally  we  are similarly going to need
that   $a_{\alpha+1}\setminus a_\alpha$ 
is not in $\mbox{triv}(F)$. 
This is handled by the properties already guaranteed
in Lemma \ref{iterseq1}. 
We need a method
to properly seal all dense sets from $V[G_{\mu_\alpha}]$.
This is handled by adding a generic  $\bar x_\alpha$
for
 $P_{\vec a\restriction \alpha}$ followed by a dominating
 real to allow us to define $a_\alpha$ and
  $x_\alpha = \bar x_\alpha\cap a_\alpha$
  (still using the properties of Lemma \ref{iterseq1}
  so as to ensure $a_\alpha\notin \mbox{triv}(F)$).
  Finally, the newest idea is 
  that we uncover some combinatorics
  to exploit given that failing to find a suitable
   $x_{\delta+2n+1}$ (or rather
    $x_{\delta+2n+1}\setminus x_{\delta+2n}$), 
    it will mean that too many names are
    failing to \textit{F-avoid} the same $\dot Y$.
 
 \bigskip

Fix an iteration sequence 
 $\langle P_\alpha, \dot Q_\beta :
 \alpha \leq \kappa, \beta < \kappa\rangle$ 
 as in  Lemma \ref{iterseq1}.
  For each $\lambda \leq\kappa$,
  let $\mathcal P(\omega,P_\lambda)$ denote
  the set of canonical $P_\lambda$-names
  of subsets of $\omega$. A name $\dot x$
  is in $\mathcal P(\omega,P_\lambda)$ if 
  $\dot x$ has the form $\bigcup_{n\in\omega} \{n\}\times A_n$
 where each $A_n$ is a countable (possibly empty) antichain
 of $P_\lambda$.
 
Suppose
 that $\dot F$ is a $P_\kappa$-name 
 that induces a non-trivial automorphism
 on $\mathcal P(\omega)/\mbox{fin}$.   
  Rather than dealing
  with a name $\dot F$, we assume instead
  that $F$ is a function from $\mathcal P(\omega,P_\kappa)$
  to $\mathcal P(\omega,P_\kappa)$ and that
   $1$ forces that $F$ induces a non-trivial
   automorphism on $\mathcal P(\omega)/\mbox{fin}$. 
 This ensures that if $p\Vdash \dot x = \dot y$,
 then $p\Vdash F(\dot x)=^* F(\dot y)$ (of course
 we could easily restrict $F$ so as to require
 that $p\Vdash F(\dot x)=F(\dot y)$). 
 
 \medskip
 
 Choose a 
   sequence $\{M_\mu : \mu<\kappa\}$
 of elementary submodels of $H(\kappa^+)$ such
 that each $M_\mu$ has cardinality less than $\kappa$,
 $F\in M_0$, 
   $M_{\mu+1}^\omega \subset M_{\mu+1}$ for all $\mu$,
   and $M_\lambda = \bigcup \{ M_\mu : \mu <\lambda\}$
   for all limit $\lambda$. 
 Let $C(F)= \{ \mu < \kappa : M_\mu\cap \kappa= \mu \}$. 
 Note that, for $\mu<\kappa$ with $\mbox{cf}(\mu)\neq\omega$,
 $F_\mu = F\cap M_\mu$ is a function from
 $\mathcal P(\omega,P_\mu)$ to 
  $\mathcal P(\omega,P_\mu)$. We also note that,
  by elementarity, $P_\mu$, for $\mbox{cf}(\mu)>\omega$,
  forces that $F_\mu$ satisfies that $\mbox{triv}(F_\mu)\cap 
   \mathcal P(A)$ is not ccc over fin for any
    $A\in\mbox{triv}(F_\mu)^+$.

\begin{lemma} 
  Assume that  $\mu\in C(F)$, $\delta<\omega_1$, and
  that\label{limitseal} 
  $\vec a\restriction\delta =
   \langle \dot a_\alpha,\dot x_\alpha : \alpha <\delta\rangle$
   is a sequence    
   of elements of $\mathcal P(\omega,P_\mu )$
   such that $1$ forces that 
    $\vec a\restriction\delta$ is an 
    $(F_\mu,\delta)$-gap. If  $\mu < \mu_1\in C(F)$, 
    then there is a pair $\{\dot a_\delta,\dot x_\delta\}
    \subset \mathcal P(\omega,P_{\mu_1})$ such that
    \begin{enumerate}
    \item $\langle \dot a_\alpha,\dot x_\alpha : \alpha <\delta{+}1\rangle$ 
    is forced by 1 to be an $(F_{\mu_1},\delta{+}1)$-gap, 
     \item $\langle \dot a_\alpha , \dot x_\alpha : \alpha <\delta{+}1\rangle$
     is forced by 1 to be a $W$-sealing extension of 
      $\langle \dot a_\alpha,\dot x_\alpha : \alpha < \delta\rangle$ 
      where $W$ is the forcing extension by $P_{\mu}$. 
    \end{enumerate}
\end{lemma}

\begin{proof}
Choose coordinates $\mu\leq \xi_0  <\mu_1$ so that  
   $\dot Q_{\xi_0}$ is equal to the 
     $P_{\mu}$-name for 
     $P_{\langle \dot a_\alpha,\dot x_\alpha : \alpha < \delta\rangle}$. 
     Then choose $\xi_0 < \xi_1 < \mu_1$ so that 
      $\dot Q_{\xi_1}$ is the standard $P_{\xi_1}$-name
      for the dominating real poset $\mathbb D$.
Pass to the forcing extension $V[G_{\xi_1}]$
and let $\bar x_{\xi_0} 
 = \bigcup \{ x_{p(\xi_0)} : ~~(p\in G_{\xi_1})\}$,
 i.e. the generic splitting of $\vec a$ added
 by $P_{\vec q}$.       
Let $\dot g_{\xi_1}$       be the standard 
 $\mathbb D$-name for the dominating real 
 added. 
 
For each $m\in \omega$, let $P_{\vec a}(m)$ be the set
of $p \in P_{\vec a}$ such that $n_p \leq m$
and $L_p\subset \{ \alpha^\delta_i  : i < m\}$. Fix any
dense set $D\subset P_{\vec a}$ with $D\in V[G_\mu]$.

\bgroup

\def\proofname{Proof of Fact:\/}

\begin{fact} For each $m\in\omega$, the set\label{oneD}
$D(m)$ is dense in $P_{\vec a}$ where
 $d\in D(m)$ if $ L_d\supset \{\alpha^\delta_i : i<m\} $,
  $m\leq n_d$, and
 for all $s\subset m$,
  $(s\cup (x_d\setminus m), n_d, L_d)\in D$.
\end{fact}

 \begin{proof}
 We remind the reader 
 of the notation $p^\downarrow s$ from 
 Proposition \ref{finitechanges}.
 Choose any $d_0 $
 in $D$ such that $n_{d_0}\geq m$
 and $L_{d_0}\supset \{ \alpha^\delta_i : i<m\}$. 
 Fix an enumeration $\{ s_i : i < 2^m\}$ 
 of $\mathcal P(m)$ so that $s_0 = x_{d_0}\cap m$.
 We recursively choose a descending sequence
  $\{ d_i  : i < 2^m\}\subset D$. 
  Having chosen $d_i$, let $s = s_i \cup (x_{d_i}\cap n_{d_i}\setminus m)$.
  Note that $d_i^\downarrow s$ is equal to 
   $(s_i\cup (x_{d_i}\setminus m), n_{d_i}, L_{d_i})$.
 By Lemma \ref{finitechanges}, $d_i^\downarrow s$ is
 an element of $P_{\vec a}$ and so we may choose
   $d_{i+1} < d_i$ such that $d_{i+1}\in D$.
    It should be clear that $d_{2^m}$ is an
    element of $D(m)$.
 \end{proof}
 
 It follows from Fact \ref{oneD} that there is a
 function $f_D\in \omega^\omega\cap V[G_{\xi_1}]$ satisfying
 that for each $m\in\omega$, there is a $d\in D(m)$
 such that $n_d < f_D(m)$,
  $L_d \subset \{ \alpha^\delta_i : i< f_D(m)\}$,
  and there is  a condition $p_d\in G_{\xi_1}$,
   such that $p_d(\xi_0) = d$. 
   Since $G_{\xi_1}$ is a filter,  
   this means that $x_d\subset \bar x_{\xi_0}$
   and $\bar x_{\xi_0}\cap a_\alpha \setminus n_d 
      = x_d\cap a_\alpha\setminus n_d$ 
      for all $\alpha\in L_d$. 
\medskip

 Let $\langle \dot k_\ell : \ell\in\omega\rangle$
 be the 
 $\mathbb D$-name of the recursively defined sequence
   $k_{\ell+1} = \dot g_{\xi_1}(k_\ell)$ (recall
  that $\mathbb D$ forces   that $\dot g_{\xi_1}$ is
  strictly increasing).
  
  \medskip

Now pass to the forcing extension $V[G_{\mu_1}]$ 
and let $\langle k_\ell : \ell\in\omega\rangle$ be
the sequence just defined.  We break into two
cases. The first case is when $\delta=\beta+1$
is a successor ordinal. Since $\omega\setminus a_\beta
\notin \mbox{triv}(F_{\mu_1})$, we can,
by Lemma \ref{notcccoverfin} and elementarity,
choose
an infinite $L\subset \omega$ in $V[G_{\mu_1}]$
so that  
that $a_{\delta} = a_\beta \cup \bigcup\{ k_{\ell+1}\setminus
 k_{\ell} : \ell\in L\}$ satisfies
 that $\omega\setminus a_\delta\notin \mbox{triv}(F_{\mu_1})$.
 Also let $x_\delta = \bar x_{\xi_0}\cap a_\delta$. 
 By the genericity of $\bar x_{\xi_0}$ it follows
 that $x_\delta\cap a_\beta =^* x_\beta$.

 \begin{fact} $\langle a_\alpha , x_\alpha : \alpha < \delta+1\rangle$
 is a $V[G_\mu]$-sealing extension of $\vec a$. 
 \end{fact}

\begin{proof} Naturally we prove that $J =\{k_\ell : \ell\in L\}$ is 
the set required in item (4) in Definition \ref{definesealing}
of sealing. 
Fix any dense set $D\subset P_{\vec a}$ with $D\in V[G_\mu]$
and choose any $m_D$ so that $f_D(m) < g_{\xi_1}(m)$ 
for all $m_D\leq m$. Fix any $p\in P_{\vec a}$ and
 $k_\ell\in J$, 
such that $n_p\leq k_\ell$ and
 $L_p\subset \{\alpha^\delta_i : i<k_\ell\}$.
 We may assume that $\beta \in \{\alpha^\delta_i : i< k_\ell\}$.
 Again, by the choice of $f_D$,
  there is a   $d\in D_(k_\ell)$ such that 
     $n_d < k_{\ell+1}$, $x_d\subset \bar x_{\xi_0}$,
      $\bar x_{\xi_0}\cap a_\alpha\setminus n_d
       = x_d\cap a_\alpha\setminus n_d$ for all
        $\alpha\in L_d$. 
By the condition on $d\in P_{\vec a}$, we also
have that $x_\alpha\setminus n_d = x_d \cap (a_\alpha\setminus n_d)$
and $a_\alpha\setminus n_d \subset a_\beta$ 
for all $\alpha\in L_d$. 
Since $k_{\ell+1}\setminus k_\ell \subset a_\delta\cap \bar x_{\xi_0}$,
it follows that $x_d\setminus x_\delta$ as required. 
The remaining requirements in item (4) of Definition
\ref{definesealing} are routine. 
\end{proof}

 \egroup

Now we consider the case that $\delta$ is a limit ordinal. 
 First, by Lemma \ref{notP} and elementarity,
  we may choose a set $\bar a_\delta$ in 
    $V[G_{\mu_1}]$ so that 
    $\omega\setminus \bar a_\delta\notin \mbox{triv}(F_{\mu_1})$
    and $a_\alpha\subset^* \bar a_\delta$ for all 
     $\alpha < \delta$.
     Then, just as in the successor case, we can choose
     an infinite $L\subset \omega$ in $V[G_{\mu_1}]$
     so that $\omega\setminus a_\delta\notin
      \mbox{triv}(F_{\mu_1})$, where
       $a_\delta = \bar a_\delta\cup \bigcup\{ k_{\ell+1}\setminus
        k_\ell : \ell \in L\}$. 
We set $x_\delta = \bar x_{\xi_0}\cap a_\delta$ and the proof
that $\langle a_\alpha, x_\alpha : \alpha < \delta+1\rangle$
is $V[G_\mu]$-sealing of 
$\langle a_\alpha , x_\alpha : \alpha <\delta\rangle$
proceeds just as it did in the successor case.    
\end{proof}

Now we turn our attention to proving there are suitable
\textit{F-avoiding\/} extensions of an $(F,\delta)$-gap.

We will make use of  a standard poset.

\begin{definition} For any ordinal $\mu<\kappa$
and indexed family\label{orthogonal}   
$\mathcal I = \{ I_\xi : \xi < \mu\} \subset [\omega]^{\aleph_0}$
such that $\mathcal I$ generates a proper ideal on 
  $\omega$, let $Q(\mathcal I)$ denote $\sigma$-centered
  poset of conditions $q = (s_q, L_q)\in  [\omega]^{<\aleph_0}
  \times [\mu]^{<\aleph_0}$ ordered by
   $r<q$ providing $s_r\supset s_q$, $L_r\supset L_q$,
   and $s_r\cap I_\xi = s_q\cap I_\xi$ for all $\xi\in L_q$.
  \end{definition}
  
  Evidently the poset $Q(\mathcal I)$ adds an infinite set
  that is almost disjoint from every element of
   $\mathcal I$ which also meets 
   every subset
   of $\omega$ that is not in the ideal generated by
    $\mathcal I$.

We will assume that $\mu_0<\mu_1$ are members of $C(F)$
with $\mu_1$ a successor. We will consider, in $V[G_{\mu_1}]$,
an 
 $(F_{\mu_1},\delta+2)$-gap 
   $\langle a_\alpha , x_\alpha : \alpha <\delta+2\rangle$
   that is a
    $V[G_{\mu_0}]$-sealing extension of 
     $\langle a_\alpha , x_\alpha : \alpha < \delta+1\rangle$.
     
For the remainder of this section we prove the main
Lemma.

\begin{lemma} There is\label{mainlemma} an $(F,\delta+3)$-gap
extending $\langle a_\alpha, x_\alpha : \alpha
 < \delta+2\rangle$ that F-avoids all
 $V[G_{\mu_0}]$-names.
\end{lemma}

\noindent
     We are working to find a $\mu_1<\mu_2\in C(F)$
     and an extending
      $(F_{\mu_2},\delta+3)$-gap 
      that is \textit{F-avoiding} for all 
      $P_{\langle a_\alpha,x_\alpha : \alpha < \delta{+}1\rangle}$-names
      $\dot Y$ in $V[G_{\mu_0}]$. 
      We may  assume   that we have chosen
      $\mu_1<\lambda_0 \in C(F)$ and
       $a_{\delta+2}\in V[G_{\lambda_0}]\setminus
        V[G_{\mu_1}]$ so
       that neither
       $a_{\delta+2}\setminus a_{\delta+1}$
       nor $\omega\setminus a_{\delta+2}$ are elements of
        $\mbox{triv}(F)$.  
   Let $\{ y_\xi : \xi < \lambda_0\} $ 
   be an enumeration of
  $[\omega]^{\aleph_0}
  \cap V[G_{\lambda_0}]$. 
      
      Let $J_{\delta}\subset\omega$ be the infinite subset of $\omega$
in item (4) of Definition \ref{definesealing} with respect to
   $\langle a_\alpha,x_\alpha : \alpha <\delta+1\rangle$ 
   being $V[G_{\mu_0}]$-sealing.
  One of the purposes
of sealing is to ensure that the value of 
 $\dot Y\cap k_{\ell}$ (for a relevant $\dot Y$)
 is determined by
  $p^\downarrow s$ for every $s\subset k_\ell$ and
  $p\in P_{\langle a_\alpha , 
  x_\alpha : \alpha < \delta+2\rangle}$ such
  that $\delta+1\in L_p$ and $n_p = k_\ell$.
 
 Let $\{ \lambda_\zeta : \zeta < \kappa\}$
  be  a cofinal subset of $C(F)\cap S^\kappa_1$  such
  that,  for all $0<\zeta<\kappa$,
   there is a strictly increasing sequence
  of successor ordinals
  $\{ \mu^\zeta_\xi : \xi < \lambda_0\}\subset 
       C(F)$ that is cofinal in 
          $\lambda_\zeta$. Assume also that
             $\lambda_\eta < \mu^\zeta_0$
             for all $\eta < \zeta <\kappa$.

For each $0<\zeta<\kappa$ and each $\xi < \lambda_0$,  
  choose a coordinate $\beta(\zeta,\xi)\in \mu^\zeta_{\xi+1}
  \setminus \mu^\zeta_\xi$ such that $\dot Q_{\beta(\zeta,\xi)}$
  is the $P_{\beta(\zeta,\xi)}$-name of the 
  Cohen poset $\mathcal C_\omega$. 
  Let $\dot c_{\beta(\zeta,\xi)}$ be the canonical name
  for the Cohen subset of $\omega$ added by $\dot Q_{\beta(\zeta,\xi)}$.
   Let $\dot y^\zeta_\xi $ be the name $y_\xi\cap 
     \dot c_{\beta(\zeta,\xi)}$. 
Let $\mathcal Y(\zeta)$ be the canonical
 $P_{\lambda_\zeta}$-name of
the indexed family
  $\{ \dot y^\zeta_\xi : \xi < \lambda_0\}$. 
     Finally, 
choose a coordinate $\bar\lambda_\zeta\in \mu^{\zeta+1}_0\setminus
 \lambda_\zeta$ so that $\dot Q_{\bar \lambda_\zeta}$ 
 is the $P_{\bar\lambda_\zeta}$-name of the poset 
  $Q(\mathcal Y(\zeta))$ from Definition \ref{orthogonal}.
  We might as well let $\dot x_{\bar \lambda_\zeta}$ 
  denote the canonical name
    for the subset added. 
    I.e. a condition $p\in P_{\kappa}$ forces
       $j\in \dot x_{\bar\lambda_\zeta}$
       so long as $p\restriction  \bar\lambda_\zeta$
       forces that $j$ is an element of the
       first coordinate of $p(\bar\lambda_\zeta)$. 
 
 We will let $\hat{x}_\zeta $ denote the canonical
 name for $
 \dot x_{\bar \lambda_\zeta}
 \cap (a_{\delta+2}\setminus a_{\delta+1})$
 and this allows us to consider if $x_{\delta+1}\cup 
 \hat{x}_\zeta$ will be a suitable choice for
    $x_{\delta+2}$.

  \begin{claim} Let $0<\zeta <\kappa$ 
  and $\xi <\lambda_0$. Suppose that\label{composed}
   $h_\xi\in V[G_{\lambda_0}]$ is a 
   1-to-1 function from $F_{\lambda_0}(y_\xi)$ into
    $\omega$ such that $h_\xi(F_{\lambda_0}(y_\xi))$
    is disjoint from $y_\xi$. Then $1$
    forces that $x^\zeta_\xi = (h_\xi\circ F)(\dot y^{\zeta}_\xi)$ is not in the ideal
    generated by the family
     $\mathcal Y(\zeta)$.
  \end{claim}
  
  \noindent 
We prove this Claim.  Since we are assuming that $F$ is forced
  to induce an automorphism on $\mathcal P(\omega)/
  \mbox{fin}$, it follows that the function
   $H_\xi$ 
   defined by $H_\xi(x) = (h_\xi\circ F)(x)$ induces
   an isomorphism from $\mathcal P(y_\xi)/
   \mbox{fin}$
   to $\mathcal P(h_\xi(F_{\lambda_0}(y_\xi)))
   /\mbox{fin}$. In addition, for all
    $\lambda_0\leq \mu\in C(F)$, 
       $H_\xi\restriction V[G_\mu]
        = (h_\xi\circ F_\mu)$  also
    induces such an automorphism. 
It is clear that $\dot y^{\zeta}_\xi$ is forced
to meet every infinite subset of $y_\xi$
that is an element of $V[G_{\mu^\zeta_\xi}]$.
It follows therefore  that $x^\zeta_\xi$ 
will meet every infinite
subset of $F(y_\xi)$ that is an element
of $V[G_{\mu^\zeta_\xi}]$. 
Similarly, for all $\eta < \lambda_0$,
  $y^\zeta_\eta$ is forced to not contain
  any infinite subset of $\omega$
  from the model $V[G_{\mu^\zeta_\eta}]$.
  Now consider any elements $\bar y_1, \bar y_2$
  of the
  ideal generated by $\mathcal Y(\zeta)$
  where $\bar y_1$ is a finite union
  from $\{ y^\zeta_\eta : \eta < \xi\}$
  and $\bar y_2 $ is the 
  union of $\{ y^\zeta_\eta : \eta \in \Gamma \}$
  for a finite 
  $\Gamma\subset \{ y^\zeta_\eta : \xi < \eta \}$.
  By the above properties, we
  have that $z_1 = F_{\lambda_0}(y_\xi)\setminus
   \bar y_1$ is forced to be 
   an infinite element of 
     $V[G_{\mu^\zeta_\xi}]$. 
Then, by recursion on $\eta\in \Gamma$,
  we have that the infinite
  set $z_1\setminus \bigcup\{
   y^\zeta_\gamma : \gamma \in \Gamma\cap\eta\}
   \in V[G_{\mu^\zeta_{\eta}}]$
   is not contained in $y^\zeta_\eta$.

  \medskip
    
    Now we assume, towards a contradiction,
    that, in $V[G_\kappa]$, for
      every subset $x\subset a_{\delta+2}\setminus a_{\delta+1}$,
    the choice $  x_{\delta+2}  = x_{\delta+1}\cup x
      $ fails, in $V[G_\kappa]$,
     to
     satisfy the forcing statement
     \[ 1 \Vdash_{P_{\langle a_\alpha,x_\alpha : \alpha <\delta+3\rangle
     }}
      \dot Y\cap F(a_{\delta+2}\setminus a_{\delta+1}) \neq^*
        F(x)\]
        for some $P_{\langle a_\alpha,x_\alpha : \alpha < \delta\rangle}$-name
        $\dot Y$    
        in $V[G_{\mu_0}]$.
        \medskip

        Given a $P_\kappa$-name $\dot x$ that is forced to be
        a subset of $a_{\delta+2}\setminus a_{\delta+1}$,
         let $\vec a_{\dot x}$ denote the $P_\kappa$-name
         of the sequence
          $\langle a_\alpha,x_\alpha : \alpha <\delta+3\rangle$
          where $ x_{\delta+2} = x_{\delta+1}\cup \dot x$. 
        In particular then, for all $0<\zeta<\kappa$, there is a
        condition $p_\zeta\in P_{\kappa} $, an integer $m_\zeta$,
        and a 
        $P_{\langle a_\alpha,x_\alpha : \alpha < \delta\rangle}$-name
        $\dot Y_\zeta$    
        in $V[G_{\mu_0}]$  such that
         $p_\zeta\restriction \lambda_0 \in G_{\lambda_0}$
         and there is a condition $r_\zeta$
         in $P_{ \vec a_{\hat{x}_\zeta}}$ such that
          $p_\zeta$ forces that 
          \[ r_\zeta
          \Vdash_{P_{\vec a_{\hat{x}_\zeta}}} 
            \left(\dot Y_\zeta \cap F(a_{\delta+2}\setminus a_{\delta+1})\right)
             \Delta \dot F(\hat{x}_\zeta) \ \ \subset m_\zeta ~~.\]

        For each $0<\zeta<\kappa$, we assume that for each $\xi<\mu$
        such that $\beta(\zeta,\xi)\in \dom(p_\zeta)$ (i.e.
        the coordinate determining the value of $\dot y^\zeta_\xi$),
        we have that $p_\zeta\restriction \beta(\zeta,\xi)$ forces
        a value on $p_\zeta(\beta(\zeta,\xi)) = (s^\zeta_\xi ,n^\zeta_\xi) 
        \in  \mathcal C_\omega$.
        Similarly, we assume that $p_\zeta\restriction \bar\lambda_\zeta$
        forces a value on $p_\zeta(\bar\lambda_\zeta)
        = (s^\zeta,L_\zeta)\in [\omega]^{<\aleph_0}
        \times [\mu]^{<\aleph_0}$. There is no loss to assume
        also that $\{ \beta(\zeta,\xi) : \xi\in L_\zeta\}
        \subset \dom(p_\zeta)$. Finally, by possibly
        increasing $m_\zeta$ and extending $p_\zeta$, we
        can assume that $s^\zeta\subset m_\zeta$
        and $n^\zeta_\xi =  m_\zeta$ for all $\xi\in L_\zeta$.
        
        We may pass to a cofinal   subset $\Gamma\subset\kappa$ 
        of such $\zeta$
        so that the sequence $\{ \dom(p_\zeta) : \zeta\in \Gamma\}$
        is a $\Delta$-system that satisfies, in addition, that
         $\dom(p_\eta)\subset \mu^\zeta_0$ for all $\eta < \zeta\in \Gamma$.
         Naturally we can also assume that there is a single name
          $\dot Y_1$ and integer $m$
          so that $\dot Y_\zeta = \dot Y_1$
          and $m_\zeta = m$ for all
           $\zeta\in \Gamma$. Next, we can assume that there is
           an integer $n_0$,  a set $s_0\subset n_0$,
           and a finite $L_0\subset \delta+2$
           such that $\{\delta, \delta+1,\delta+2\}\subset L_0$ 
         and $r_\zeta$ is equal to
           $(s_0\cup (x_{\delta+1}\setminus n_0)\cup
           (\hat{x}_\zeta\setminus n_0), n_0, L_0)$.
        
 Now we may choose $\zeta_1 < \zeta_2$ from $\Gamma$ so that
  $s^{\zeta_1}=s^{\zeta_2}$, $L_{\zeta_1}=L_{\zeta_2}$, 
  and for all $\xi\in L_{\zeta_1}$, $s^{\zeta_1}_\xi
   = s^{\zeta_2}_\xi$.
 
         \medskip
          
          Now let $\hat{x}^1 = \hat{x}_{\zeta_1}$, 
            $\hat{x}^2 = \hat{x}_{\zeta_2}$ and
             let $\hat{x}^3 $ be the canonical name
             for $\hat{x}^1\cup \hat{x}^2$. 
             Finally choose a condition $q_3 \in P_{\kappa}$ extending
              $p_{\zeta_1}$ and $p_{\zeta_2}$ 
        so that there is an
        integer $\bar m$, 
        a $P_{\langle a_\alpha,x_\alpha:\alpha < \delta\rangle}$-name
        $\dot Y_3$ in $V[G_{\mu_0}]$ and a condition
          $r_3$ in $P_{\vec a_{\hat{x}^3}}$
          so that $q_3$ forces that 
           \[ r_3
          \Vdash_{P_{\vec a_{\hat{x}^3 }}} 
            \left(\dot Y_3 \cap F(a_{\delta+2}\setminus a_{\delta+1})\right)
             \Delta \dot F(\hat{x}^3  ) \ \ \subset \bar m  ~~\]
             and
             \[ \left(\dot F(\hat{x}^1)\cup \dot F(\hat{x}^2) \right)\Delta 
                 \dot  F(\hat{x}^3) \subset \bar m~~.\]
                  Again, there is no loss to assuming
                  that $\{\delta,\delta+1,\delta+2\}\subset L_{r_3}$
                   and let $n_3 = n_{r_3}$ and $s_3 = x_{r_3}\cap n_3$,
                   and $x_{r_3} = s_3\cup (x_{\delta+1}
                   \setminus n_3)\cup (\hat{x}^3\setminus n_3)$.
        \medskip

We continue our analysis in the model $V[G_{\lambda_0}]$. 
Consider the condition $q_3$. We again assume
 that, for $i=1,2$,
 $q_3\restriction \bar\lambda_{\zeta_i}$ 
 forces a value $q_3(\bar\lambda_{\zeta_i}) =
 (\bar s_i , \bar L_i)\in 
    Q(\mathcal Y(\zeta_i))$.
Also, that $\{ \beta(\zeta_i,\xi) : \xi \in \bar L_i\}
 \subset \dom(q_3)$ for $i=1,2$, and, as with $p_{\zeta_i}$,
  for each $\xi\in \bar L_i$, 
   $q_3\restriction \beta(\zeta_i,\xi)$ 
   forces a value on $q_3(\beta(\zeta_i,\xi))$ that 
   has the form $(s, n(\beta(\zeta_i,\xi))$.
 By possibly again
 increasing the value of $\bar m$, we can assume that
 $(\bar s_1\cup\bar s_2)\subset\bar m$
 and $n(\beta(\zeta_i,\xi))\leq\bar m$ 
 for all $\xi\in \bar L_i$ and $i=1,2$.
 Set $\bar s_3 = \bar s_1\cup \bar s_2$.

Consider the sequence $\mathcal D = 
 \langle D_k : k \in\omega\rangle \in V[G_{\mu_0}]$
 where $d\in D_k\subset P_{\langle a_\alpha, x_\alpha :
  \alpha < \delta\rangle}$ providing
  $d$ forces a value on 
  each of $\dot Y_1\cap k$
  and $\dot Y_3\cap k$.
By possibly increasing $\bar m$,
we can assume that for
 all $\bar m < \ell\in J_{\delta}$ there are  
 integers $  k^0_\ell,k^1_\ell $ such that
 for all
  $s\subset \ell$, 
the condition $(s\cup (x_{\delta+1}\setminus \ell), k^0_\ell, 
   \{\delta{+}1\}\cup
   \{\alpha^{\delta+1}_i : i<k^1_\ell\})$ forces a value
   on each of $\dot Y_1\cap \ell$
   and $\dot Y_3\cap \ell$.
   For simplicity assume that $\bar m < \min(J_{\delta})$.

  \begin{claim} for any finite $s_1,s_2\subset (\omega\setminus a_\delta)$
 such that $\bar m \leq \min(s_1\cup s_2)$ and $\max(s_1\cup s_2)\leq \ell\in 
  J_{\delta}$, there is a $q_{s_1,s_2}\leq q_3$ forcing that
  \begin{enumerate}
   \item $\hat{x}^1\cap \ell  = \bar s_1\cup s_1$,
   \item $\hat {x}^2\cap \ell = \bar s_2\cup s_2$,
   \item $\hat{x}^3\cap \ell = \bar s_3\cup (s_1\cup s_2)$.
  \end{enumerate}
 \end{claim}

Let $B_\delta$ denote the set $F(a_{\delta+1}\setminus a_\delta)$
 (which is in the model $V[G_{\lambda_0}]$).
For each $i=1,2,3$, $\ell\in J_{\delta}$ and $s\subset \ell$,
 let $H_s^i(\ell)$ denote the value forced on 
   $\dot Y_1\cap B_\delta \cap (\ell\setminus \bar m )$
   (for $i=1,2$) and $\dot Y_3\cap B_\delta\cap (\ell\setminus \bar m)$
   for $i=3$ by the condition 
 $((\bar s_i\cup s)\cup (x_{\delta+1}\setminus \ell), k^0_\ell, 
   \{\delta{+}1\}\cup
   \{\alpha^{\delta+1}_i : i<k^1_\ell\})$, 
   $i=1,2,3$ respectively. Notice that this is simply
   a property of the names $\dot Y_1$ and $\dot Y_3$ 
   and does not depend on the properties of $\dot F$.
   However, using the connections to $\dot F$ we have
   this next critical claim.
 
     \begin{claim} For all $\ell\in L_\delta$ and
        $s_1,s_2\subset (a_{\delta+2}\setminus a_{\delta+1})\cap
        (\ell\setminus \bar m)$, \\
\centerline{          \( H^1_{s_1}(\ell)\cup H^2_{s_2}(\ell) = H^3_{s_1\cup s_2}(\ell)\).
          }\end{claim}
          
\noindent          This is  because, in addition to  the conditions on 
         \( H^1_{s_1}(\ell), H^2_{s_2}(\ell), H^3_{s_1\cup s_2}(\ell)\) in
         connection to $\dot Y_1$ and $\dot Y_3$ mentioned above,
            $q_{s_1,s_2}$ forces that 
            \begin{enumerate}
            \item  $\dot F(\hat{x}^1) \cap B_\delta\cap (\ell\setminus \bar m) =H^1_{s_1}(\ell)$,
            \item  $\dot F(\hat{x}^2) \cap   B_\delta\cap (\ell\setminus \bar m) =H^2_{s_2}(\ell)$,
            \item  $\dot F(\hat{x}^3) \cap   B_\delta\cap (\ell\setminus \bar m) =H^3_{s_3}(\ell)$,
            \item $\left(\dot F(\hat{x}^1)\cup \dot F(\hat{x}^2)\right) \cap B_\delta\setminus \bar m$
            is equal to $\dot F(\hat{x}^3)\cap B_\delta\setminus \bar m$.
            \end{enumerate}
          
          \begin{claim} For all $\ell\in J_\delta$, each of
           $H^1_\emptyset(\ell)$ and $H^2_\emptyset(\ell)$ are forced by $q_3$
           to be subsets of
              $\dot F(\hat{x}^3)$.
          \end{claim}
          
  \noindent  In the forcing extension $V[G_\kappa]$, we have that for
  each $s = \hat{x}^3\cap \ell$ ($\ell\in J_\delta$), 
   $H^3_s(\ell) = F(\hat{x}^3)\cap B_\delta \cap (\ell\setminus \bar m)$.
   We also have that \\
   \centerline{$H^3_s(\ell) = H^1_s(\ell)\cup H^2_\emptyset(\ell)
    = H^1_\emptyset(\ell)\cup H^2_s(\ell)$. }
    \medskip
    
    It will be convenient to simply assume that
    $\bar m$ was chosen large enough so
    that $H^1_\emptyset(\ell)\cup
     H^2_\emptyset(\ell)$ is a subset of $\bar m$ for all
      $\ell\in L_\delta$. In fact, by the definition 
      of $H^i_s(\ell)$, this means that $H^1_\emptyset(\ell)$
      and $H^2_\emptyset(\ell)$ are empty for all $\ell\in L_\delta$.

          \begin{claim}
          It is forced by $1$ that  $\hat{x}^1$ contains no infinite
          set from $V[G_{\lambda_0}]$ and  that $\hat{x}^2$ contains
          no infinite set from $V[G_{\mu^{\zeta_1}_0}]$.
          Similarly, therefore, $\dot F(\hat{x}^1)$ is
          forced by $1$ to contain  
          no infinite set from $V[G_{\lambda_0}]$ and
            $\dot F(\hat{x}^2)$ contains no infinite
            set from $V[G_{\mu^{\zeta_1}_0}]$.
            Finally, we also have that $1$ forces
            that $\dot F(\hat{x}^3)$ contains no infinite
            set from $V[G_{\lambda_0}]$.
          \end{claim}

\begin{claim}
For all $\ell\in J_\delta$ and   $s\subset (\omega\setminus a_{\delta+1})\cap
(\ell\setminus \bar m)$, 
the value of $H^1_{s}(\ell)$ is equal to the value 
of $H^1_{s\cap a_{\delta+2} }(\ell)$.
\end{claim}

\noindent Let $s' = s\cap (a_{\delta+2}\setminus a_{\delta+1})$
and consider a generic filter $G_\kappa$ with $q_{s',s'}\in G_\kappa$.
Then  consider  the forcing statement
  \[ (s_0\cup (x_{\delta+1}\setminus n_0) \cup
    (\hat{x}^1\setminus n_0) , n_0, L_0)\Vdash 
       \dot Y_1\cap B_\delta\setminus \bar m 
       =
F(\hat{x}^1)\setminus \bar m\]
in the model $V[G_\kappa]$. By the assumption 
 that $p_{\zeta_1}\Vdash r_{\zeta_1}\in P_{\langle 
    a_\alpha,x_\alpha : \alpha <\delta+3\rangle}$,
  we have that $s'\subset \hat{x}^1\setminus n_0$
 and $s\setminus s'$ is disjoint from $s_0\cup
 x_{\delta+1}\cup \hat{x}^1$.
 Moreover, 
 $ (s_0\cup (s\setminus s')
 \cup (x_{\delta+1}\setminus n_0) \cup
    (\hat{x}^1\setminus n_0) , n_0, L_0)$
 is an extension of
  $
   (s_0\cup (x_{\delta+1}\setminus n_0) \cup
    (\hat{x}^1\setminus n_0) , n_0, L_0)$
    because $a_\alpha\setminus a_{\delta+1}\subset n_0$
    for all $\alpha\in L_0$. 
    This shows that the value of $H^1_s(\ell)$
    is indeed the same as that of $H^1_{s'}(\ell)$.

    Let us also note  that if $r$ is any extension of
     $ (s_0\cup (x_{\delta+1}\setminus n_0) \cup
    (\hat{x}^1\setminus n_0) , n_0, L_0)$
    in $P_{\langle a_\alpha,x_\alpha : \alpha <\delta+3\rangle}$,
    then $x_r \setminus (x_{\delta+1}\cup \hat{x}^1)$
    is contained in $\omega\setminus a_{\delta+1}$.

 \begin{claim} For all finite 
 $s\subset (a_{\delta+2}\setminus a_{\delta+1})$,
 the  sequence $\{ H^1_s(\ell) : s \subset\ell\in J_{\delta}\}$
 is eventually constant. 
 \end{claim}
 
 \noindent It follows from the fact that $r\uparrow \ell_2$ is an
 extension of $r$ if $n_r\leq\ell_1$ that for $s\subset \ell_1<\ell_2$,
    $H^1_{s}(\ell_1)\subset H^1_{s}(\ell_2)$. 
 Just as was the case  with $s=\emptyset$, it follows
 that $q_{s,\emptyset}$ forces that $H^1_s(\ell)$
 is a subset of $F(\hat{x}^3)$ for all $s\subset\ell\in J_\delta$.
 Since $\bigcup\{ H^1_s(\ell) : s\subset\ell\in J_\delta\}$ is
 an element of $V[G_{\lambda_0}]$ and is a subset
 of $F(\hat {x}^3)$, it follows that this union must be finite.
 
 \medskip
 
 Now that we know that $\{ H^1_s(\ell) : s\subset \ell\in L_\delta\}$
 is eventually constant for all 
  $s\subset a_{\delta+2}\setminus (a_{\delta+1}\cup\bar m) $,
   there is no loss to assuming that $H^1_s(\ell) = H^1_s(\ell')$
   for all $s\subset \ell<\ell'$ with $\ell,\ell'\in J_\delta$.
   Clearly, by symmetry, we can also
   arrange to having  the same conditions holding
   for $H^2_s(\ell)$ for such $s$ and $\ell\in J_\delta$.
   For these reasons we shall henceforth let
     $H^1_s$ and $H^2_s$ denote the value 
     of $H^1_s(\ell)$ and $H^2_s(\ell)$ respectively
     where $\ell$ is the minimum in $J_\delta$ with $s\subset \ell$.
     Similarly, the value of $H^3_s$ can simply be defined
     as $H^1_s\cup H^2_s$.
     
   \medskip
   
   \begin{claim} For all finite $s\subset (a_{\delta+2}\setminus a_{\delta+1})$,
      $H^1_s = H^2_s =H^3_s$.
      \end{claim}
      
      \noindent 
      Recall that for all $s\subset \ell \in L_\delta$, we have
      seen that $H^3_s(\ell) = H^1_s(\ell)\cup H^2_\emptyset(\ell)$
      and $H^3_s(\ell) $ is also equal to 
        $H^1_\emptyset(\ell)\cup H^2_s(\ell)$. Of
        course with our updated assumption on $H^1_\emptyset(\ell)$
        and $H^2_\emptyset(\ell)$ being empty, this
        shows that $H^1_s = H^3_s = H^2_s$ as claimed.
        \medskip

\begin{claim} For all finite $s,t\subset (a_{\delta+2}\setminus a_{\delta+1})$,
 $H^1_s\cup H^1_t = H^1_{s\cup t}$. 
\end{claim}
   
   \noindent    We have that $H^1_t = H^2_t$ and
     $H^1_s\cup H^2_t = H^3_{s\cup t} = H^1_{s\cup t}$.
     
 For each $j\in (a_{\delta+2}\setminus a_{\delta+1})$, 
 let $H^1_j = H^1_{\{j\}}$. 
It is now immediate, but worth noting,  that

\begin{claim} For all finite $s\subset (a_{\delta+2}\setminus a_{\delta+1})$,
   $H^1_s = \bigcup_{j\in s} H^1_j$.
\end{claim}

\noindent The claim follows easily  by induction on $|s|$. 
   \medskip
   
   \begin{claim} In\label{onto} $V[G_\kappa]$ with $q_3\in G_\kappa$,\\
\centerline{      $F(\hat{x}^1) \setminus \bar m$
      is equal to   $\bigcup \{ H^1_j : j\in \hat{x}^1\setminus \bar m\}$.}
      \end{claim}
      
\noindent This unlikely Claim follows from the forcing assertion
\[ (s_0\cup (x_{\delta+1}\setminus n_0)\cup
 (\hat{x}^1\setminus n_0), n_0, L_0)
\Vdash \dot Y_1\cap B_\delta\setminus \bar m
  = F(\hat{x}^1)\]
  and the fact that, for all $\ell\in J_\delta$,
  the condition \\
  \(
  (s_0\cup (x_{\delta+1}\setminus n_0)\cup
 (\hat{x}^1\cap (\ell\setminus n_0)), k^0_\ell , 
  \{ \delta+1,\delta+2\}\cup \{\alpha^{\delta+1}_i : i< k^1_\ell\})\)\\
    decides the value of
\( \dot Y_1\cap B_\delta\cap (\ell \setminus \bar m)   
  \)
 \bigskip
 
Let $\bar a = a_{\delta+2}\setminus a_{\delta+1}$.   

\begin{claim} $B_\delta\subset^* \bigcup \{ j \in \bar a : H^1_j\}$
\end{claim}

\noindent Otherwise, choose any infinite $z\subset \bar a$ in 
 $V[G_{\lambda_0}]$, 
so that $F(z) =^* B_\delta \setminus \bigcup\{ H^1_j : j\in \bar a\}$.
 It is routine to check that $\hat{x}^1$ is forced to meet
 $z$ in an infinite set and therefore
   $F(\hat{x}^1)$ must meet $F(z)$ 
   in an infinite set. 
   This however contradicts Claim \ref{onto}.
      
\begin{claim} For each $\tilde m\in \omega$,
 the\label{nonempty}
 set $z_{\tilde m} = \{ j\in \bar a : H^1_j\subset \bar m\}$ is finite.
 \end{claim}
 
 \noindent Assume otherwise and recursively choose an infinite
 sequence of pairs $\{ j_i , k_i : i\in \omega\}$ 
 so that $k_i \in F(z_{\tilde m})\cap H^1_{j_i}
   \setminus \bigcup\{ H^1_{j_n} : n<i\}$, 
   with $j_i\in \bar a$. This we may do since
    $F(z_{\tilde m})\subset^* B_\delta$. 
Choose $\xi <\mu$ so that $F(y_\xi) = \{ k_i : i\in \omega\}$.
Consider the infinite set $y^{\zeta_1}_\xi \subset y_\xi$.
Set $x^{\zeta_1}_\xi$ to be the set
  $\{ j_i : k_i \in F(y^{\zeta_1}_\xi)\}$. 
   We note that $y_\xi\subset^* z_{\tilde m}$ 
   and $\{j_i : i\in\omega\}\cap z_{\tilde m}$ is empty.
   Therefore $x^{\zeta_1}_\xi $ and $y^{\zeta_1}_\xi$
   are disjoint. By Claim \ref{composed},
     $x^{\zeta_1}_\xi$ is not contained
     in any finite union from the family
      $\mathcal Y(\zeta_1)$. This implies
      that $\hat{x}^1\cap x^{\zeta_1}_\xi$ is infinite
      and yet $\hat{x}^1\cap y^{\zeta_1}_\xi$ is finite.
      This is a contradiction since $F(\hat{x}^1)\cap
       F(y^{\zeta_1}_\xi)$ is not finite because
        $F(\hat{x}^1)$ contains $H^1_{j_i}$
        (and therefore $k_i$)
        for all
        infinitely many $j_i\in \hat{x}^1$. 
  
  \begin{claim} There is $\tilde m\in \omega$ such
  that the elements of the family $\{ H^1_j \setminus
  \bar m: j\in \bar a\}$ are pairwise disjoint.
  \end{claim}
  
  \noindent If this Claim fails to hold then we may
  recursively choose  a sequence $\{ i, j^1_i, j^2_i : i\in I\}$
  for some infinite set $I\subset \bar a$ so that for each $i<i'\in I$
  \begin{enumerate}
  \item  
  $j^1_i\neq j^2_i$ and $\max(\{i,j^1_i,j^2_i\})
   < \min(\{i', j^1_{i'}, j^2_{i'}\})$
   \item  
   $i\in H^1_{j^1_i}\cap H^1_{j^2_i}$.
   \end{enumerate}
   Since $F (\{ j^1_i : i \in I\})$ and
     $F (\{j^2_i : i\in I\})$ are almost disjoint, we
     may suppose that $I_2 = I\setminus 
     F(\{j^1_i  :i\in I\})$ is 
     infinite. Choose $\xi < \mu$ so
      $F(y_\xi) =^* I_2$. It follows that $y_\xi\cap 
      \{j^1_i : i\in I\}$ is finite. 
 Now given $y^{\zeta_1}_\xi$ we let
   $x^{\zeta_1}_\xi = \{ j^1_i : i\in F(y^{\zeta_1}_\xi)\}$.   
  We again have by Claim \ref{composed},
  that $x^{\zeta_1}_\xi $ is not in the ideal
  generated by $\mathcal Y(\zeta_1)$. 
While $\hat{x}^1$ is almost disjoint from $y^{\zeta_1}_\xi$,
 we have that $F(\hat{x}^1) $ contains the
 infinite set $\{ i \in F(y^{\zeta}_\xi) : j^1_i\in 
    \hat{x}\cap x^{\zeta_1}_\xi\}$.
  
   \begin{claim} There\label{oneone}
   is a $\tilde m\in\omega$
   such that the family $\{ H^1_j \setminus \tilde m : 
   \tilde m<j\in \bar a\}$ are disjoint singleton sets.
   \end{claim}
   
   \noindent The failure of this Claim,
   together with the conclusion of Claim \ref{nonempty},
   would have that
   there is a strictly increasing sequence
     $\{j_k : k\in \omega\}$ so that, for each
      $k\in \omega$, $H^1_{j_k} $
      contains a pair $\{i^1_k, i^2_k\}$ with
       $\max(\bigcup\{ H^1_{j_\ell} : \ell<  k\})
        < \min(\{ i^1_k , i^2_k\})$.
Given that the pair
 $F^{-1}(\{ i^1_k : k\in\omega\})$
 and $F^{-1}(\{i^2_k : k\in \omega\})$
 are almost disjoint, 
 we may assume that $\{ j_k : k\in \omega\}
 \setminus 
 F^{-1}(\{ i^1_k : k\in\omega\})$ is infinite.
 Let $I_2 = \{k : j_k \notin F^{-1}(\{i^1_k : k\in \omega\})\}$
 and 
 choose $y_\xi$ so that
  $F(y_\xi) =^* \{ i^1_k : k\in I_2\}$. 
        We may choose $y_\xi$ to be disjoint
        from $x_\xi = \{ j_k : k\in I_2\}$.
   Now, with $y^{\zeta_1}_\xi\subset y_\xi$ 
   as given above, 
   set $x^{\zeta_1}_\xi =  \{ j_k : i^1_k \in F(y^{\zeta_1}_\xi)
   \}$.
   Again, Claim \ref{composed} ensures that $x^{\zeta_1}_\xi$   is not in the ideal
   generated by $\mathcal Y(\zeta_1)$. Again this 
   leads to a contradiction from
   the facts that $\hat{x}^1\cap x^{\zeta_1}_\xi$ 
   is infinite while $F(\hat{x}^1)\cap F(y^{\zeta_1}_\xi)$
   is finite. This is because each $j_k\in 
    \hat{x}^1\cap x^{\zeta_1}_\xi$ 
    gives rise to another value $i^1_k\in 
    F(\hat{x}^1)\cap F(y^{\zeta_1}_\xi)$.
   
   \medskip
   
   Now that we have established the above claims, 
    we choose $\tilde m$ as
    in Claim \ref{oneone} and
    let $h_{\bar a}$ be a 
     bijection between a cofinite subset
     of $\bar a$ and $B_\delta\setminus \tilde m$
     such that $h_{\bar a}(j) \in H^1_j\setminus
      \tilde m$ for all $j\in \dom(h_{\bar a})$. 
      We complete the proof of Lemma \ref{mainlemma}
      by obtaining a final contradiction.
  Since $\bar a\notin\mbox{triv}(F)$, 
  we can choose in $V[G_{\lambda_0}]$, 
  by Proposition \ref{badx},
  an infinite $z\subset \dom(h_{\bar a})$ such that
    $h_{\bar a}(z)\cap F(z)$ is finite. 
    Notice that $h_{\bar a}(x) =
    \bigcup\{ H^1_j \setminus \tilde m : j\in x\}$
    for all $x\subset \dom(h_{\bar a})$.

    Choose $\xi<\mu$ so that $F(y_\xi)=^*
       h_{\bar a}(z)$ and note that
         we may assume that
          $y_\xi\cap z$ is empty.   
Let $x^{\zeta_1}_\xi = 
 \{ j \in z : F(y^{\zeta_1}_\xi) \cap
   H^1_j\setminus \tilde m\neq \emptyset\}$.
      Since $x^{\zeta_1}_\xi\subset z$ we 
   have that $x^{\zeta_1}_\xi $ and $y^{\zeta_1}_\xi$
   are disjoint. As in the previous
   cases, $x^{\zeta_1}_\xi$ is not in the
   ideal generated by $\mathcal Y(\zeta_1)$. 
   It follows that $F(\hat{x}^1)\supset H^1_j$
   for infinitely many $j\in x^{\zeta_1}_\xi$,
   which in turn implies that
    $F(\hat{x}^1)\cap F(y^{\zeta_1}_\xi)$ is infinite.
    This contradicts that $\hat{x}^1\cap y^{\zeta_1}_\xi$
    is finite.

\section{Completing the proof}

\begin{theorem}
Assume that  $\kappa>\omega_1$ 
is a   
regular cardinal, GCH holds below $\kappa$,
and that there is a $\diamondsuit(S^\kappa_1)$-sequence.
Then there is a finite support ccc iteration sequence
 $\langle P_\xi , \dot Q_\eta : \xi\leq \kappa,
  \eta <\kappa\rangle$ satisfying that in the
  forcing extension by $P_\kappa$,
   all automorphisms of $\mathcal P(\omega)/\mbox{fin}$
   are trivial.
\end{theorem}

\begin{proof}
Fix the $\diamondsuit(S^\kappa_1)$-sequence
 $\{ D_\lambda : \lambda\in S^\kappa_1\}$.
Assume that
 $\langle P_\xi , \dot Q_\eta : \xi\leq \kappa,
  \eta <\kappa\rangle$
  satisfies all the conditions
  detailed in Lemma \ref{iterseq1}. 
We supplement the requirements on the iteration as
follows. Consider any $\lambda\in S^\kappa_1$.
If $D_\lambda$ is a $P_\lambda$-name,
rename it $\dot F_\lambda$, that is forced by
 $1$ to induce a non-trivial automorphism
 on $\mathcal P(\omega)/\mbox{fin}$, 
 then choose, if it is forced by $1$
 to be possible, 
 an $(F_\lambda,\omega_1)$-gap
   $\langle  a_\alpha, x_\alpha : \alpha < \omega_1\rangle$
   in $V[G_\lambda]$ such that
   \begin{enumerate}
   \item the poset $P_{\langle a_\alpha,x_\alpha :
    \alpha <\omega_1\rangle}$ is ccc, and
\item the resulting $\omega_1$-gap
   $\langle F_\lambda(a_\alpha),
   F(x_\alpha) : \alpha < \omega_1\rangle$
    is not split in the forcing extension
    by $P_{\langle a_\alpha,x_\alpha :
    \alpha <\omega_1\rangle}$.
   \end{enumerate}
 In such a case, then $\dot Q_{\lambda}$
 is the $P_\lambda$-name of the
 poset $P_{\langle a_\alpha, x_\alpha : 
  \alpha < \omega_1\rangle}$, and
     $\dot Q_{\lambda+1}$ is 
     the $P_{\lambda+1}$-name of the
     poset 
      $Q( \langle F(a_\alpha), 
       F(x_\alpha) : \alpha < \omega_1\rangle)$
       as identified in Proposition
        \ref{freeze}.   
    
    Now we verify that with this completed
    definition of $\langle P_\xi , \dot Q_\eta 
    : \xi\leq\kappa, \eta<\kappa\rangle$ that
    the forcing extension has the desired property.
    We recall that, as per Lemma \ref{iterseq1},
    $P_\kappa$
    forces that $\omega^\omega$-cohere
    and P-cohere will hold.
    
Let $F = \{ s_\xi : \xi\in\kappa\} $ be a 
sequence of pairs from $\mathcal P(\omega,P_\kappa)$
such that $F$, when viewed as a    
$P_\kappa$-name of
a function from  $\mathcal P(\omega)$ 
to $\mathcal P(\omega)$ is forced, by $1$
to   induce    a non-trivial
 automorphism on  
  $\mathcal P(\omega)/\mbox{fin}$. 
  For simplicity we assume that
  every element of $\mathcal P(\omega,P_\kappa)$
  appears as the first coordinate of $s_\xi$
  for some $\xi < \kappa$.
  
   Let $\varphi\in \kappa^\kappa$ be 
   the function satisfying that 
      $ s_\xi  = d_{\varphi(\xi)}$
      and let $A = \{ \varphi(\xi) : \xi\in \kappa\}$.     
 Choose a 
   sequence $\{M_\mu : \mu<\kappa\}$
 of elementary submodels of $H(\kappa^+)$ such
 that each $M_\mu$ has cardinality less than $\kappa$,
 $F,A,\varphi \in M_0$, 
   $M_{\mu+1}^\omega \subset M_{\mu+1}$ for all $\mu$,
   and $M_\lambda = \bigcup \{ M_\mu : \mu <\lambda\}$
   for all limit $\lambda$. 
 Let $C(F)= \{ \mu < \kappa : M_\mu\cap \kappa= \mu \}$. 
 Note that, for $\mu<\kappa$ with $\mbox{cf}(\mu)\neq\omega$,
 $F_\mu = F\cap M_\mu$ is a function from
 $\mathcal P(\omega,P_\mu)$ to 
  $\mathcal P(\omega,P_\mu)$. 
  We recall that, as per Lemma \ref{iterseq1},
    $P_\kappa$
    forces that $\omega^\omega$-cohere
    and P-cohere will hold, and so 
  by elementarity, $P_\mu$, for $\mbox{cf}(\mu)>\omega$,
  forces that $F_\mu$ satisfies that $\mbox{triv}(F_\mu)\cap 
   \mathcal P(A)$ is not ccc over fin for any
    $A\in\mbox{triv}(F_\mu)^+$.

Fix any $\lambda\in C(F)\cap S^\kappa_1$ 
such that $A_\lambda$ is equal to $A\cap  \lambda$.
Naturally it follows from elementarity 
that $D_\lambda$
is equal to $\{ s_\xi : \xi\in \lambda\}
 = F_\lambda = F\restriction \mathcal P(\omega,P_\lambda)$.
 
First  pass to the extension $V[G_\lambda]$ and assume
 that it was possible to choose an
  $(F_\lambda,\omega_1)$-gap as described in (1) and (2)
  above. There is a canonical 
  $P_{\lambda+1}$-name $\dot Y_\lambda$
  for the subset of $\omega$ 
  that splits the gap $\langle a_\alpha, x_\alpha 
   :\alpha < \omega_1\rangle$. Choose the
  $\zeta\in\kappa$ such that 
    $\dot Y_\lambda$ is the first coordinate
    of the pair $s_\zeta$ and let $\dot Z_\lambda$
    denote the second coordinate
    of $s_\zeta$. Now pass to the extension
     $V[G_\kappa]$. 
    By the assumption
    on $\dot Q_{\lambda+1}$, the set
     $Z_\lambda = \val_{G_\kappa}(\dot Z_\lambda)$
     does not (and can not) split the gap
       $\{ F(a_\alpha), F(x_\alpha) : \alpha <
        \omega_1\}$. That is, there is
        an $\alpha<\omega_1$ such
        that $Z_\lambda \cap F(a_\alpha) \neq^*
         F(x_\alpha)$. However,
          $Y_\lambda = \val_{G_\kappa}(\dot Y_\lambda)$
          does satisfy that $Y_\lambda
           \cap a_\alpha =^* x_\alpha$. 
           This implies that $F(Y_\lambda) = Z_\lambda$
           fails to satisfy the requirements
           of  an automorphism
           of $\mathcal P(\omega)/\mbox{fin}$.
           
           \medskip
           
Now we prove that there is indeed such a 
required $(F_\lambda,\omega_1)$-gap as above. 
Fix any strictly increasing sequence $\{ \mu_\alpha
 : \alpha < \omega_1\} \subset C(F)\cap \lambda$
 consisting of successor ordinals. By recursion
 on $\alpha\in\omega_1$ we choose a pair
   $a_\alpha, x_\alpha\in [\omega]^{\aleph_0}\cap 
      V[G_{\mu_{\alpha}}]$.  We $x_0\subset a_0\subset
      \omega$ in $V[G_{\mu_0}]$ simply so
      that $\omega\setminus a_0\in \mbox{triv}(F)^+$.
Recall      that the conclusions of
Lemma \ref{notP} and Lemma \ref{notcccoverfin}
hold in $V[G_\lambda]$ for $F_\lambda$.

      The inductive assumptions
      are, for each limit $\delta<\omega_1$ and integer
       $n\in \omega$ are:
      \begin{enumerate}
      \item $\langle a_\alpha , x_\alpha : \alpha < \delta+n\rangle$
        is an $(F_\lambda,\delta+n)$-gap in $V[G_{\mu_{\delta+n}}]$,
    \item $\langle a_\alpha , x_\alpha : \alpha < \delta+1\rangle$
    is a $V[G_{\mu_{\delta}}]$-sealing extension of
      $\langle a_\alpha , x_\alpha : \alpha < \delta\rangle$,
      \item $\langle a_\alpha , x_\alpha : \alpha < \delta+2n+2\rangle$
       $F_\lambda$-avoids all $V[G_{\mu_\delta+2n}]$-names,
      \item $\langle a_\alpha, x_\alpha : \alpha < \delta+2n+3\rangle$
      is a $V[G_{\mu_{\delta+2n+2}}]$-sealing extension of
      $\langle a_\alpha, x_\alpha : \alpha < \delta+2n+2\rangle$.
      \end{enumerate}
           
 Assume that $\delta$ is a limit and that we have
 chosen the sequence $\langle   a_\alpha , x_\alpha : \alpha < \delta\rangle$
 satisfying the inductive hypotheses. 
 It follows from Lemma \ref{limitseal} that 
 a pair $a_\delta, x_\delta$ exists in $V[G_{\mu_\delta}]$ as
 required. 
 
 Now assume that $n\in\omega$ and that we have
 chosen the sequence 
  $\langle   a_\alpha , x_\alpha : \alpha < \delta+2n+1\rangle$
 satisfying the inductive hypotheses. 
 It follows from Lemma \ref{mainlemma}, and elementarity,
 that there is a pair $a_{\delta+2n+1}, x_{\delta_2n+1}$
 in $V[G_{\mu_{\delta+2n+2}}]$ 
 so that $\langle a_\alpha , x_\alpha : 
  \alpha < \delta + 2n+2\rangle $ is
  an $(F_\lambda, \delta+2n+2)$-gap
  that F-avoids all $V[G_{\mu_{\delta+2n}}]$-names.
  Similarly, it follows again from Lemma \ref{limitseal}
  that there is a pair
   $a_{\delta+2n+2}, x_{\delta+2n+2}$ in 
     $V[G_{\mu_{\delta+2n+3}}]$ 
     such that
  $\langle   a_\alpha , x_\alpha : \alpha < \delta+2n+3\rangle$
    is a $V[G_{\mu_{\delta+2n+2}}]$-sealing
    extension of
  $\langle   a_\alpha , x_\alpha : \alpha < \delta+2n+2\rangle$.
  
  \medskip
  
  This completes the proof that the recursion can be completed.
  It follows from Lemma \ref{cccsplitgap} that
    $P_{\langle a_\alpha, x_\alpha : \alpha\in \omega_1\rangle}$
    is ccc.
    It follows from Lemma \ref{cccfreezegap}
    that $P_{\langle a_\alpha, x_\alpha : \alpha\in \omega_1\rangle}$
    forces that 
      $\langle F_\lambda(a_\alpha), F_\lambda(x_\alpha) 
       : \alpha <\omega_1\rangle$ is an
        $\omega_1$-gap that is not split.
    
 This completes the proof.
\end{proof}

\end{document}